\documentclass[a4paper,english,reqno]{amsart}

\usepackage{amssymb,amsthm, amsmath,amsfonts}
\usepackage{bbm}
\usepackage{graphicx}
\usepackage{subfigure}
\usepackage{microtype}
\usepackage{booktabs}
\usepackage{enumerate}
\usepackage{pgf}
\usepackage[latin1]{inputenc}
\usepackage{verbatim}
\usepackage{cite}
\usepackage{url}
\usepackage{xcolor}

\usepackage[cmtip,all]{xy}
\xyoption{arc}

\usepackage{pgf,tikz}
\usetikzlibrary{arrows,shapes,calc,positioning}
\usepackage{manfnt}
\usepackage{tikz-cd}
\usepackage{url}
\usepackage{microtype}

\usepackage[font=footnotesize]{caption}

\newcommand{\F}{\mathcal{F}}

\newcommand{\R}{\mathbb{R}}
\newcommand{\Z}{\mathbb{Z}}

\renewcommand{\H}{\mathbb{H}}

\newcommand{\C}{\mathbb{C}}
\newcommand{\T}{\mathbb{T}}
\newcommand{\Sp}{\mathbb{S}}

\renewcommand{\O}{\mathcal{O}}

\newcommand{\matriz}[4]{\left(
\begin{array}{cc}
 #1 & #2 \\
 #3 & #4
\end{array}
\right)}

\newcommand{\inte}[1]{\mathaccent23{#1}}

\newcommand{\scirc}{{\scriptstyle \circ}}

\newtheorem{theorem}{Theorem}

\newtheorem{proposition}{Proposition}

\theoremstyle{definition}
\newtheorem{definition}[theorem]{Definition}
\newtheorem{example}[theorem]{Example}

\theoremstyle{remark}

\newtheoremstyle{named}{}{}{\itshape}{}{\bfseries}{.}{.5em}{\thmnote{#3} #1}
\theoremstyle{named}

\numberwithin{equation}{section}

\usepackage[pagewise]{lineno}

\begin{document}

\title[On the construction of minimal foliations by hyperbolic surfaces ]{On the construction of minimal foliations by hyperbolic surfaces on 3-manifolds}


\author[F. Alcalde]{Fernando Alcalde Cuesta}
\address{Instituto de Matem\'aticas, Universidade de Santiago de Compostela, E-15782, Santiago de Compostela, Spain. }
\email{fernando.alcalde@usc.es}
\thanks{This work was partially supported by Spanish MINECO/AEI Excellence Grant MTM2016-77642-C2-2-P, Galician Grant GPC2015/006 and European Regional Development Fund, ANII (Uruguay) FCE-135352 and project PAPIIT IN106817 (UNAM, Mexico). The authors Fran\c coise Dal'Bo, Matilde Mart\'{\i}nez, and Alberto Verjovsky would like to thank the University of Santiago de Compostela for its hospitality. All the authors are also grateful to the referees for a very careful reading of the manuscript and many valuable suggestions.}

\author[F. Dal'Bo]{ Fran\c{c}oise Dal'Bo}
\address{Institut de Recherche Math\'ematiques de Rennes, Universit\'e de  Rennes 1, F-35042 Rennes, France}
\email{francoise.dalbo@univ-rennes1.fr}

\author[M. Mart\'{\i}nez]{Matilde Mart\'{\i}nez}
\address{Instituto de Matem\'atica y Estad\'{\i}stica Rafael Laguardia, Facultad de Ingenier\'{\i}a,
Universidad de la Rep\'ublica, J.Herrera y Reissig 565, C.P.11300 Montevideo, Uruguay.}
\email{matildem@fing.edu.uy}

\author[A. Verjovsky]{Alberto Verjovsky}
\address{Universidad Nacional Aut\'onoma de M\'exico,
Apartado Postal 273, Admon. de correos \#3, C.P. 62251 Cuernavaca,
Morelos, Mexico.}
\email{alberto@matcuer.unam.mx}

\subjclass[2010]{Primary 57R30}

\date{}



\begin{abstract}
We describe several methods to construct  minimal
foliations by hyperbolic surfaces on closed 3-manifolds, and discuss the properties of the examples thus obtained. 
\end{abstract}


\maketitle


\section{Introduction and motivation} \label{Sintro}

Foliations by surfaces generalize compact surfaces in many ways. 
Poincar\'e's Uniformization Theorem tells us that  {\em most} compact surfaces are hyperbolic --namely, those with negative Euler characteristic.
In a similar way, many (arguably {\em most}) 
foliations by surfaces are {\em hyperbolic} in the sense that they admit a Riemannian metric along
the leaves with constant curvature $-1$.
In the 1960's, 
Lickorish proved that any closed orientable 
3-manifold admits a foliation by surfaces,
but in his construction there is always a toral leaf that bounds a Reeb component.
Some explicit descriptions of foliations by hyperbolic surfaces appear in the works of Calegari, Fenley, and Gabai, who, motivated by the work of Thurston, used foliations and laminations in the study of 3-manifolds.
\medskip 

On the other hand, the geodesic and horocycle flows over compact hyperbolic surfaces have
been studied in depth since the work by 
Hopf and 
Hedlund in the 1930's. The minimality of the horocycle flow for foliations and laminations by hyperbolic surfaces has been recently dealt with in several papers by the authors and 
Matsumoto, see \cite{AD, ADMV, MMV, Matsumoto}. 
Specifically, in \cite{ADMV}, we showed the following dichotomy for any foliation $\F$ by dense hyperbolic surfaces on a closed manifold $M$:
{\em Either all leaves of $\F$ are geometrically infinite (i.e. with non finitely generated fundamental group) or the leaves without holonomy are simply connected and all essential loops represent non trivial holonomy.} In the former case (where all leaves have infinitely many ends, infinite genus, or both), we proved that the horocycle flow on the unit tangent bundle of $\F$ is also \emph{minimal }(i.e. all its leaves are dense). 
\medskip 

Thus, from this double perspective, we are interested in knowing which 3-manifolds admit foliations by hyperbolic surfaces, and we ask ourselves: {\em Given one of the eight Thurston geometries, is there a manifold $M$, admitting such a geometric structure, and having a foliation $\F$ by hyperbolic surfaces? Of which type?}
As proved by 
Yue in \cite[Theorem 1]{Yue}, if a compact 3-manifold $M$ admits a foliation by surfaces of negative curvature, then its fundamental group $\pi_1(M)$ must have exponential growth. Therefore, Yue's theorem implies that any closed 3-manifold admitting a geometric structure of type $\Sp^3$, $\Sp^2\times \R$, $\R^3$ or $Nil$ does not admit a foliation by hyperbolic surfaces since its fundamental group is finite or infinite with polynomial growth. On the contrary, for each of the geometric structures $Sol$, $\H^2\times\R$, $\widetilde{SL}(2,\R)$ and $\H^3$, there are manifolds admitting minimal foliations by hyperbolic surfaces. 
\medskip 

Any closed orientable 3-manifold modeled on the solvable group $Sol$ is the quotient $\T^3_A$ of the product $\T^2 \times \R$ by the $\Z$-action generated by $f(x,y,t) = (A(x,y), t+1)$, where $A$ is a linear hyperbolic automorphism of $\T^2$. The suspension of $A$ gives an Anosov flow on $\T^3_A$, which is induced by the  vector field $\partial / \partial t$. The center-unstable foliation $\F$ of this flow --namely, the foliation induced by the product of the unstable foliation $\F_A^+$ associated to the eigenvalue $|\lambda| >1$
and $\R$-- is a foliation by hyperbolic surfaces. As we shall see
in Theorem~\ref{thm:Novikov}, since the fundamental group of $\T^3_A$ is solvable, it does not admit a foliation with geometrically infinite leaves.
\medskip 

As for 3-manifolds which are either modeled on $\H^2\times\R$ or $\widetilde{SL}(2,\R)$,
we shall show in Proposition~\ref{prop:suspension} that both types of hyperbolic surfaces, geometrically finite and infinite, can be realized as leaves of minimal foliations using the classical suspension method from the representation of a Fuchsian group. In Section~\ref{Sbasic}, we apply three different methods, also classical in some sense, to construct 
other
examples of minimal foliations by geometrically infinite surfaces on this kind of manifolds, or some which do not admit any geometric structure.
\medskip 

In the core of the paper, we deal with the construction of minimal foliations by hyperbolic surfaces on hyperbolic 3-manifolds. Foliations by hyperbolic surfaces are known to exist on closed hyperbolic 3-manifolds, see \cite{Calegaribook}, \cite{Fenley} and \cite{Thurston3}. Examples without compact leaves have been constructed 
in \cite{Meigniez} and  \cite{Nakayama}. Minimal foliations by geometrically finite surfaces also appear as center-stable foliations of transitive Anosov flows, which are known to exist on hyperbolic manifolds. Here, we present two families of minimal examples with geometrically infinite leaves and different transverse dynamics. 
\medskip 

In Section~\ref{Swithoutholonomy}, we construct foliations without holonomy. These can only exist in fiber bundles over $\Sp^1$, see \cite{Tischler}. In \cite{ConlonGrowth}, Cantwell and  Conlon proved that the leaves of any $C^2$ foliation without holonomy are homeomorphic to the plane, the cylinder or the hyperbolic surfaces with one or two nonplanar ends, see Figure~\ref{fig:leaves}. The only closed 3-manifold admitting a $C^2$ foliation by planes is $\T^3$ \cite{Rosenberg}, whereas the closed 3-manifolds admitting a $C^2$ foliation by cylinders are nilmanifolds \cite{Hector}. We prove: 
\medskip 

\begin{theorem} \label{thm:withoutholonomy}  Let $M$ be a closed orientable hyperbolic 3-manifold that fibers over $\Sp^1$ with fiber $\Sigma$. Assume that the first Betti number $b_1 \geq 2$. Then $M$ admits a minimal foliation without holonomy $\F$  whose leaves are diffeomorphic to an abelian regular cover of $\Sigma$. 
\end{theorem} 

In Section~\ref{Sbranching}, we describe another method to obtain minimal foliations by geometrically infinite surfaces on hyperbolic 3-manifolds from the center-(un)stable foliation of the Anosov flow on $T^3_A$ via branched coverings. More precisely, we prove: 

\begin{theorem} \label{thm:Franks}
 Let $M$ be a closed orientable hyperbolic 3-manifold that fibers over $\Sp^1$. Assume that the monodromy has oriented unstable manifolds and the stretch factor is quadratic over $\mathbb{Q}$. Then $M$ admits a minimal transversely affine foliation $\F$ whose leaves are geometrically 
 infinite surfaces.
\end{theorem}

In both cases, up to a homeomorphism preserving the fibration,  the monodromy (and hence the manifold and the foliation) can be always assumed of class $C^\infty$. 
Finally, in Section~\ref{SAnosov}, we illustrate this construction via branched coverings with two examples done with ``bare  hands''. 
\medskip 

All the examples we construct are {\em $\R$-covered} since  
 the leaf spaces of the lifted foliations to $\H^3$ are diffeomorphic to $\R$. Theorem~\ref{thm:withoutholonomy} provides foliations without holonomy which are {\em uniform}  \cite{Thurston3} since any two leaves in $\H^3$ are a bounded (in fact, constant) distance apart. On the contrary, all the foliations obtained from Theorem~\ref{thm:Franks}  are not uniform since they have the same transverse structure as the unstable foliation of the Anosov flow on $T^3_A$, see 
 \cite
 {Calegaribook} and 
 \cite
 {Thurston3}. 
Using the {\em continuous extension property} from \cite{Fenley}, it is possible to adapt the argument of \cite{CannonThurston} describing
 the {\em universal circle} of $\F$ and proving that it is a {\em  $2$-sphere-filling Peano curve}.

\section{Some preliminaries} \label{Spre}

\subsection{Foliations}
A {\em foliation by surfaces} $\F$ on a closed 3-manifold $M$ is given by a $C^\infty$ atlas $\{(U_i,\varphi_i)\}$ which consists of $C^\infty$ diffeomorphisms $\varphi_i: U_i \to D\times T$ from open sets $U_i$ that cover $M$ to the product of an open disk $D$ in $\R^2$ and an open interval $T$ in $\R$ such that the  change of coordinates  
has the form $\varphi_i \circ \varphi_j^{-1}(x,y) = (\varphi_{ij}(x,y), \gamma_{ij}(y))$ for every $(x,y) \in \varphi_j(U_i \cap U_j)$.
Each $U_i$ is called a {\em foliated chart}, a set $\varphi_i^{-1}(\{x\}\times T)$ being its
{\em transversal}. The sets of the form $\varphi_i^{-1}(D\times\{y\})$, called {\em plaques}, glue together to form maximal
connected surfaces called {\em leaves}. A foliation is said to be {\em minimal} if all its leaves are dense.  
The tangent bundle $T\F$ and the normal bundle $N\F = TM / T\F$ are 
naturally trivialized on each foliated chart. 
The foliation $\F$ is {\em tranversely orientable} if $N\F$ is 
orientable. 
In this paper, we will always consider transversely orientable $C^\infty$ foliations on closed oriented 3-manifolds. 
\medskip 

In each foliated chart we can endow the plaques with a Riemannian metric, in a continuous way. Gluing the local
metrics with partitions of unity gives a Riemannian metric on each leaf, which varies continuously in the $C^\infty$ topology.
Each leaf is then endowed with a conformal structure, or equivalently, with a Riemann surface structure. According to 
\cite{Candel} and \cite{Verjovsky}, if all leaves are uniformized by the 
Poincar\'e half-plane $\H^2$, the uniformization map is continuous and 
leaves have constant curvature $-1$. We say
that $\F$ is a foliation {\em by hyperbolic surfaces}. Recall that a hyperbolic surface $S$ is the quotient of 
$\H^2$ under the action of a torsion-free discrete subgroup
$\Gamma$ of the group $PSL(2,\R)$ of orientation preserving isometries of $\H^2$. 
When $\Gamma$ is of finite type we say that $S$ is {\em geometrically finite}, else it is {\em geometrically infinite}.  
\medskip

The following theorem summarizes some important results by 
Haefliger \cite{Haefliger} and 
 Novikov \cite{Novikov}, as well some improvements by Rosenberg \cite{Rosenberg} and  Hector  \cite{Hector}:

\begin{theorem} \label{thm:Novikov} Let $\F$ be a foliation by surfaces of a compact 3-manifold $M$. Assume $\F$ does not contain a Reeb component. Then the following properties are satisfied: 

\begin{list}{\labelitemi}{\leftmargin=0pt}
\item[~(1)]$M$ is irreducible if $\F$ is not defined by a $\Sp^2$-fibration over $\Sp^1$.

\item[~(2)]  Every leaf $L$ of $\F$ is {\em  incompressible}, i.e. if $i:L\hookrightarrow M$ denotes the inclusion, the induced map $i_*:\pi_1(L)\to  \pi_1(M)$ is injective.

\item[~(3)] Any closed transversal $\gamma$ represents a nontrivial and non-torsion element of the fundamental group of $M$, and therefore, $\pi_1(M)$ is infinite.

\item[~(4)] If $\pi_1(M)$ 
 does not contain a free group with two generators, then all leaves are planes or cylinders, and $\F$ has both planar
 and cylindrical leaves when they are hyperbolic. 
\end{list} 
\end{theorem}

\subsection{Geometrization of surface bundles and 3-manifolds}

Let $\Sigma$ be a closed orientable surface of genus $g \geq 1$, and let $\varphi : \Sigma \to \Sigma$ be an orientation-preserving homeomorphism. 
If $g=1$, then $\varphi$ is isotopic to a linear automorphism of the torus $\T^2$.
This is given by the $2 \times 2$ matrix 
\begin{equation} \label{sl(2,Z)}
\hspace{5em}A = \matriz{a}{b}{c}{d} \in SL(2,\Z)
\end{equation}
of the automorphism $\varphi_\ast : \pi_1(\T^2) \to \pi_1(\T^2)$ with respect to the canonical basis $m$ and $p$ of 
$\pi_1(\T^2) = \Z \oplus \Z$.  It is well-known the dynamical behavior of $A$ is related to $| \, tr(A) \, |  = | a+d |$. 
Indeed, if  $|\,tr(A)\,| < 2$, $A$ has finite order. If $|\,tr(A)\,| = 2$,   up to multiplication by $- I$, $A$ is conjugated to a matrix of the form
\begin{equation} \label{nilmatrix}
\matriz{1}{n}{0}{1}
\end{equation}
that fixes $m=(1,0)$. Finally, if $|\,tr(A)\,| > 2$, up to multiplication by $- I$, $A$ has eigenvalues $\lambda >1$ and $1/\lambda < 1$ which correspond to eigenvectors $u^{\pm}$. The foliations $\F_A^{\pm}$ generated by 
$u^{\pm}$ are preserved by  $A$, but the leaves of $\F_A^+$ and $\F_A^-$ are stretched by a factor of $\lambda$ and 
$1/\lambda$ respectively.
\medskip 

If $g \geq 2$, according to a theorem of Thurston \cite{Thurston1}, 
the first two cases also appear, although the third case is much more subtle. Namely, up to isotopy, one of the following alternative holds: 
\begin{list}{\labelitemi}{\leftmargin=17pt}
\item[(1)] $\varphi$ is {\em periodic}, i.e. there is an integer $n \geq 1$ such that $\varphi^n$ is the identity $I$.
\item[(2)] $\varphi$ is {\em reducible}, i.e. $\varphi$ preserves the union of finitely many disjoint essential simple closed 
curves in $\Sigma$.
\item[(3)] $\varphi$ is {\em pseudo-Anosov},  i.e. there is a pair of transverse singular foliations $\F_\varphi^{\pm}$ endowed 
with transverse invariant measures $\mu^{\pm}$ having no atoms and full support and there is a real number 
$\lambda >1$ such that the foliations $\F_\varphi^{\pm}$ are invariant by $\varphi$ but the measures 
$\varphi_\ast (\mu^+) = \lambda \mu^+$ and $\varphi_\ast (\mu^- )=  
(1/\lambda) \mu^-$. 
\end{list}
\medskip 

Let $M_\varphi$  be the mapping torus given as the quotient of $\Sigma \times [0,1]$ by the equivalence relation that identifies $(x,0)$ with $(\varphi(x),1)$. 
By construction, it has a natural fibration
$f: M_\varphi \to \Sp^1$
induced by the trivial fibration $p_2 : \Sigma \times [0,1] \to [0,1]$. 
If $g =1$, 
$M_\varphi$ satisfies one the following conditions: 
\begin{list}{\labelitemi}{\leftmargin=17pt}
\item[(1)] If $\varphi$ is periodic, then $M_\varphi$ admits a $\R^3$ geometric structure. 
\item[(2)] If $\varphi$ is reducible, then $M_\varphi$ contains an incompressible torus.
\item[(3)] If $\varphi$ is Anosov, then $M_\varphi$ admits a $Sol$ geometry. 
\end{list}
If $g \geq 2$, Thurston's geometrization theorem for surface bundles \cite{Thurston2} states: 
\begin{list}{\labelitemi}{\leftmargin=17pt}
\item[(1)] If $\varphi$ is periodic, then $M_\varphi$ admits an $\H^2 \times \R$ geometry.
\item[(2)] If $\varphi$ is reducible, then $M_\varphi$ contains an incompressible torus. 
\item[(3)] If $\varphi$ is pseudo-Anovov, then $M_\varphi$ admits an $\H^3$ geometry. 
\end{list}
\medskip

Thurston's geometrization conjecture was proved by Perelman and says that every prime closed 
3-manifold $M$, 
either has a geometric structure 
or it splits along incompressible tori into pieces which admit a geometric structure and whose interiors have finite volume. 
The eight geometric structures described by Thurston derive from the simply connected manifolds $\Sp^3$, $\R^3$, and $\H^3$, the product manifolds $\Sp^2\times \R$ and $\H^2\times\R$, and the Lie groups $Nil$, 
$Sol$, and $\widetilde{SL}(2,\R)$. In an equivalent way, each prime closed 3-manifold $M$ either has a geometric structure or 
it splits along a family of disjoint incompressible tori as the union of {\em hyperbolic manifolds} 
(which admit $\H^3$ geometries) and Seifert fibered pieces. 

\subsection{Foliated $\Sp^1$-bundles}

Let $\Sigma$ be a closed orientable surface of genus $g \geq 2$. Assume $\Sigma = \Gamma \backslash \H^2$ is given by the action of a discrete subgroup $\Gamma$ of 
$PSL(2,\R)$.
The product $\Sigma\times \Sp^1$ is an example of foliation by hyperbolic surfaces, but more interesting examples can be constructed as suspension of a representation $\rho:\Gamma\to PSL(2,\R)$.
Since $\rho(\Gamma)$ acts on
$\Sp^1=\partial \mathbb{H}^2$, the circle at infinity of $\mathbb{H}^2$, we have a diagonal action of $\Gamma$ on $\H^2\times \Sp^1$ giving rise to a $\Sp^1$-bundle $M=\Gamma \backslash (\H^2\times \Sp^1)$ over $\Sigma$. The horizontal foliation of $\H^2\times \Sp^1$ induces a foliation $\F$ whose leaves are diffeomorphic to the hyperbolic surfaces $\Gamma_z \backslash \H^2$. Here,
$\Gamma_z=\{\gamma\in\Gamma: \ \rho(\gamma)(z)=z \}$ is the stabilizer of $z \in \Sp^1$.
The minimality of $\F$ is equivalent to that of the action of $\Gamma$ on $\Sp^1$. Foliated $\Sp^1$-bundles are classified, up to $C^\infty$ isomorphism, by their {\em Euler class} $e(\rho)$, which is an integer such that $|e(\rho)| \leq 2g-2$. As remarked in \cite{Scott},  $M$ admits a geometric structure $\H^2\times \Sp^1$ or $\widetilde{SL}(2,\R)$ depending on whether $e(\rho) = 0$ or $e(\rho) \neq 0$, 
and we can state the following result:

\begin{proposition}  \label{prop:suspension}
Any foliated $\Sp^1$-bundle $M$ has a minimal foliation by geometrically finite hyperbolic surfaces. If 
$|e(\rho)|\neq 2g-2$, 
it also has a minimal foliation with geometrically infinite leaves.
\end{proposition}
\begin{proof}
Consider $\rho:\Gamma\to PSL(2,\R)$ and let $e(\rho)$ be its Euler class. As seen in \cite{Goldman1}, if $|e(\rho)|=2g-2$ then $\rho$ is discrete and faithful. In this case, $M$ is the unit tangent bundle of $\Sigma$ and $\F$ is the center-stable foliation for the geodesic flow of $\Sigma$, which is a minimal foliation by planes and cylinders.
All the other Euler classes admit both faithful and non-faithful representation, as proved in \cite{deBlois-Kent} and \cite{FunarWolff}. 
If $|e(\rho)|\neq 2g-2$, a faithful representation $\rho$ is never discrete, so it is always minimal. If $\rho$ is faithful, then the stabilizer $\Gamma_z$ of any point $z \in \Sp^1$ is solvable, either trivial or infinite cyclic, and therefore the leaf passing through $z$ is either a plane or a cyclinder. This proves that all foliated bundles admit minimal foliations by planes and cylinders. Besides, all representations of nonnull Euler class correspond to minimal actions, see \cite[Theorem 2]{Ghys}, and when 
$|e(\rho)|\neq 0, 2g-2$, non-faithful representations define minimal foliations with leaves of infinite type, by a result of \cite{ADMV} mentioned earlier.
Finally, by \cite[Theorem 2]{Wood2}, if $\rho$ factorizes through a free group or a free abelian group, then $e(\rho)=0$.
In this case, $\rho$ has a nontrivial kernel, and under the assumption that $\rho(\Gamma)$ acts minimally on $\Sp^1$ (including for instance an elliptic element of infinite order), the leaves 
we obtain are geometrically infinite. 
\end{proof}

\setcounter{equation}{0} 

\section{Minimal foliations by geometrically infinite surfaces: \\ variations on some classical constructions} \label{Sbasic}

 In this section,
we construct minimal foliations by geometrically infinite surfaces in 3-manifolds with 
geometry $\H^2\times\R$, but also without geometric structure.

\subsection{Cut-and-paste constructions}

We start by describing a classical procedure for gluing handles to leaves of a foliation.
Let $\F_0$ be a minimal foliation of codimension one of a closed manifold $M_0$. We shall modify $(M_0,\F_0)$ in order to construct another minimal foliated manifold $(M,\F)$ whose leaves are geometrically infinite. Firstly, since  $\F_0$ is minimal there is a closed 
transversal $\gamma_0$ which meets all leaves. Let $V_0 \cong D \times \Sp^1$  be a closed neighborhood of $\gamma_0$ such that the 
foliation induced by $\F_0$ is conjugated to the horizontal foliation by disks $D \times \{z\}$. 
The manifold $M_1 = M_0 - \inte{V}_0$ obtained by removing the interior of $V_0$ admits a foliation $\F_1$ which is transverse to the boundary 
$\partial M_1 \cong \Sp^1 \times \Sp^1$. By construction, the foliation induced by $\F_1$ on $\partial M_1$ is conjugated to the 
horizontal foliation 
by circles $\Sp^1 \times \{z\}$. 
Let $\Sigma$ be the closed surface of genus $g \geq 1$, and remove the interior of a small disk $D$. The horizontal foliation $\F_2$ of
$M_2 = (\Sigma - \inte{D}) \times \Sp^1$ induces 
the same foliation on
$\partial M_2 \cong \Sp^1 \times \Sp^1$. If we glue together $(M_1,\F_1)$ and 
$(M_2,\F_2)$ using an orientation-preserving $C^\infty$ diffeomorphism $\psi : \Sp^1 \times \Sp^1 \to \Sp^1 \times \Sp^1$ respecting the horizontal foliation  by  \lq meridians\rq, we obtain a foliated closed manifold $(M,\F)$ whose leaves are geometrically infinite without
planar ends (i.e. all neighborhoods of all ends have infinite genus). 

\begin{definition} 
We say that $(M,\F)$ is obtained by \emph{Dehn surgery} on a knotted or unknotted closed 
transversal $\gamma_0$. We can also replace $\gamma_0$ with a link $\gamma = \gamma_1 \cup \dots \cup \gamma_k$ 
transverse to $\F_0$
\end{definition}

\begin{example} \label{ex:surgery1}
Take a linear foliation $\F_0$ on $\T^3$ given as the suspension of a representa\-tion 
$\rho : \pi_1(\T^2) \to \rm{Diffeo}_+^\infty(\Sp^1)$ sending $m$ and $p$ to rotations $R_\alpha(z)=e^{2\pi i\alpha}z$ and 
$R_\beta(z) = e^{2\pi i\beta}z$ where $\alpha$ and $\beta$ are linearly independent over $\Z$. By Dehn surgery on a closed transversal $\gamma_0$, 
the leaves of $\F$ become {\em infinitely many infinite Loch Ness monsters} or {\em double Jacob's ladders}, see Figure~\ref{fig:leaves}.
Now, let us discuss if $M$ admits some geometry or not by distinguishing several cases:
\end{example}

\begin{list}{\labelitemi}{\leftmargin=17pt}
\item[(i)\;] Assume that $\gamma_0$ is a fiber of the trivial fibration from $M_0 = \T^2 \times \Sp^1$ onto $\T^2$ (up to isotopy), and then $\partial V_0$ is fibered by \lq parallels\rq. If the gluing map $\psi$ preserves the product structure of $\partial M_1$ and $\partial M_2$, then $M$ has a structure of trivial $\Sp^1$-bundle over $\T^2 \# \Sigma$ having a $\H^2 \times \R$ geometry. However, if $\psi$ does not preserve the fibration by  \lq parallels\rq, then $M$ is a \emph{graph-manifold}, i.e. having a nontrivial toral
decomposition into Seifert fibered pieces. Indeed, both pieces $M_1$ and $M_2$ admit Seifert fibered structures inducing the same Seifert fibered structure on their boundaries $\partial M_1 \cong \Sp^1 \times \Sp^1$ and $\partial M_2 \cong \Sp^1 \times \Sp^1$. But this structure is not preserved by $\psi$. By gluing $\partial M_1$ and $\partial M_2$, we obtain an incompressible torus $T$ in $M$. Up to isotopy, we can assume that $T$ belongs to a family of incompressible tori defining a toral decomposition of $M$. The torus $T$ is the common boundary of two Seifert fibered pieces. But since $\partial M_1$ and $\partial M_2$ admit a unique Seifert fibered structure (up to isotopy) induced by the Seifert fibered structure of $M_1$ and $M_2$ \cite[Lemma 1.15]{Hatcher}, we deduce that $M$ is not a Seifert fibered manifold. 

\item[(ii)] In general, $\gamma_0$ remains transverse to a linear foliation $\mathcal{H}_0$ by toral leaves, a perturbation of $\F_0$. By changing coordinates of $M_0$, we can assume that $\mathcal{H}_0$ is the horizontal foliation of  $M_0 = \T^2 \times \Sp^1$. There is a $1$-dimensional foliation $\mathcal{L}_0$ transverse to $\mathcal{H}_0$ which admits $\gamma_0$ as leaf. If we pick a leaf of $\mathcal{H}_0$, homeomorphic to $\T^2$, then $\mathcal{L}_0$ is the suspension of a homeomorphism $\varphi$ of the torus $T^2$ with $r$ marked points given from the intersection of $\gamma_0$. Therefore the piece $M_1$ admits a Seifert fibered structure, a nontrivial toral decomposition, or a hyperbolic structure depending on the mapping class of $\varphi$ is periodic, reducible, or pseudo-Anosov. In the last case, the manifold $M$ consists of a hyperbolic piece and a Seifert piece, and does not admit any geometric structure. 
If $M_1$ has a Seifert fibered structure, as in the case (i), the manifold $M$ is a graph-manifold where the gluing map $\psi$ does not preserve the Seifert fibered structure of $\partial M_1$ and $\partial M_2$. But 
if $\psi$ preserves the Seifert fibered structure of $\partial M_1$ and $\partial M_2$, then $M$ admits a Seifert fibered structure (having a $\H^2 \times \R$ geometry). In fact,
as $\F$ is a foliation without holonomy transverse to the fibration, its Euler class is zero.  Indeed, we can apply \cite[Theorem 3.4]{EHN}) to the transverse fibrations over $S^1$ that approach $\F$, or directly factorize the holonomy representation (in the sense of \cite[Theorem 3.5]{EHN}) through $SO(2)$.
\end{list}

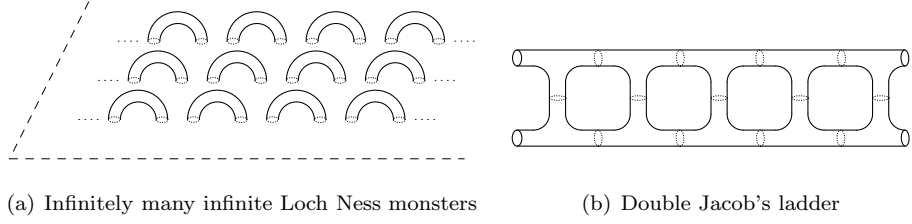
\begin{figure}[t]
\subfigure[Infinitely many infinite Loch Ness monsters]{
\begin{tikzpicture}[line cap=round,line join=round,>=triangle 45,x=1.0cm,y=1.0cm, scale=0.26]
\clip(-4,-5) rectangle (19.5,4.5);
\draw [line width=0.2pt,dash pattern=on 3pt off 3pt] (0,4)-- (-4,-4);
\draw [line width=0.2pt,dash pattern=on 3pt off 3pt] (-4,-4)-- (19,-4);
\draw [shift={(2.5,-2)}] plot[domain=0:pi,variable=\t]({1*1.5*cos(\t r)+0*1.5*sin(\t r)},{0*1.5*cos(\t r)+1*1.5*sin(\t r)});
\draw [shift={(2.51,-2.01)}] plot[domain=-0.02:3.13,variable=\t]({1*0.91*cos(\t r)+0*0.91*sin(\t r)},{0*0.91*cos(\t r)+1*0.91*sin(\t r)});
\draw [rotate around={0.36:(1.3,-2)},dash pattern=on 0.1pt off 0.8pt] (1.3,-2) ellipse (0.3cm and 0.14cm);
\draw [rotate around={0.36:(3.71,-2.01)},dash pattern=on 0.1pt off 0.8pt] (3.71,-2.01) ellipse (0.3cm and 0.14cm);
\draw [shift={(6.5,-2)}] plot[domain=0:pi,variable=\t]({1*1.5*cos(\t r)+0*1.5*sin(\t r)},{0*1.5*cos(\t r)+1*1.5*sin(\t r)});
\draw [shift={(6.51,-2.01)}] plot[domain=-0.02:3.13,variable=\t]({1*0.91*cos(\t r)+0*0.91*sin(\t r)},{0*0.91*cos(\t r)+1*0.91*sin(\t r)});
\draw [rotate around={0.36:(5.3,-2)},dash pattern=on 0.1pt off 0.8pt] (5.3,-2) ellipse (0.3cm and 0.14cm);
\draw [rotate around={0.36:(7.71,-2.01)},dash pattern=on 0.1pt off 0.8pt] (7.71,-2.01) ellipse (0.3cm and 0.14cm);
\draw [shift={(10.5,-2)}] plot[domain=0:pi,variable=\t]({1*1.5*cos(\t r)+0*1.5*sin(\t r)},{0*1.5*cos(\t r)+1*1.5*sin(\t r)});
\draw [shift={(10.51,-2.01)}] plot[domain=-0.02:3.13,variable=\t]({1*0.91*cos(\t r)+0*0.91*sin(\t r)},{0*0.91*cos(\t r)+1*0.91*sin(\t r)});
\draw [rotate around={0.36:(9.3,-2)},dash pattern=on 0.1pt off 0.8pt] (9.3,-2) ellipse (0.3cm and 0.14cm);
\draw [rotate around={0.36:(11.71,-2.01)},dash pattern=on 0.1pt off 0.8pt] (11.71,-2.01) ellipse (0.3cm and 0.14cm);
\draw [shift={(14.5,-2)}] plot[domain=0:pi,variable=\t]({1*1.5*cos(\t r)+0*1.5*sin(\t r)},{0*1.5*cos(\t r)+1*1.5*sin(\t r)});
\draw [shift={(14.51,-2.01)}] plot[domain=-0.02:3.13,variable=\t]({1*0.91*cos(\t r)+0*0.91*sin(\t r)},{0*0.91*cos(\t r)+1*0.91*sin(\t r)});
\draw [rotate around={0.36:(13.3,-2)},dash pattern=on 0.1pt off 0.8pt] (13.3,-2) ellipse (0.3cm and 0.14cm);
\draw [rotate around={0.36:(15.71,-2.01)},dash pattern=on 0.1pt off 0.8pt] (15.71,-2.01) ellipse (0.3cm and 0.14cm);
\draw [shift={(3.5,0)}] plot[domain=0:pi,variable=\t]({1*1.5*cos(\t r)+0*1.5*sin(\t r)},{0*1.5*cos(\t r)+1*1.5*sin(\t r)});
\draw [shift={(3.51,-0.01)}] plot[domain=-0.02:3.13,variable=\t]({1*0.91*cos(\t r)+0*0.91*sin(\t r)},{0*0.91*cos(\t r)+1*0.91*sin(\t r)});
\draw [rotate around={0.36:(2.3,0)},dash pattern=on 0.1pt off 0.8pt] (2.3,0) ellipse (0.3cm and 0.14cm);
\draw [rotate around={0.36:(4.71,-0.01)},dash pattern=on 0.1pt off 0.8pt] (4.71,-0.01) ellipse (0.3cm and 0.14cm);
\draw [shift={(7.5,0)}] plot[domain=0:pi,variable=\t]({1*1.5*cos(\t r)+0*1.5*sin(\t r)},{0*1.5*cos(\t r)+1*1.5*sin(\t r)});
\draw [shift={(7.51,-0.01)}] plot[domain=-0.02:3.13,variable=\t]({1*0.91*cos(\t r)+0*0.91*sin(\t r)},{0*0.91*cos(\t r)+1*0.91*sin(\t r)});
\draw [rotate around={0.36:(6.3,0)},dash pattern=on 0.1pt off 0.8pt] (6.3,0) ellipse (0.3cm and 0.14cm);
\draw [rotate around={0.36:(8.71,-0.01)},dash pattern=on 0.1pt off 0.8pt] (8.71,-0.01) ellipse (0.3cm and 0.14cm);
\draw [shift={(11.5,0)}] plot[domain=0:pi,variable=\t]({1*1.5*cos(\t r)+0*1.5*sin(\t r)},{0*1.5*cos(\t r)+1*1.5*sin(\t r)});
\draw [shift={(11.51,-0.01)}] plot[domain=-0.02:3.13,variable=\t]({1*0.91*cos(\t r)+0*0.91*sin(\t r)},{0*0.91*cos(\t r)+1*0.91*sin(\t r)});
\draw [rotate around={0.36:(10.3,0)},dash pattern=on 0.1pt off 0.8pt] (10.3,0) ellipse (0.3cm and 0.14cm);
\draw [rotate around={0.36:(12.71,-0.01)},dash pattern=on 0.1pt off 0.8pt] (12.71,-0.01) ellipse (0.3cm and 0.14cm);
\draw [shift={(15.5,0)}] plot[domain=0:pi,variable=\t]({1*1.5*cos(\t r)+0*1.5*sin(\t r)},{0*1.5*cos(\t r)+1*1.5*sin(\t r)});
\draw [shift={(15.51,-0.01)}] plot[domain=-0.02:3.13,variable=\t]({1*0.91*cos(\t r)+0*0.91*sin(\t r)},{0*0.91*cos(\t r)+1*0.91*sin(\t r)});
\draw [rotate around={0.36:(14.3,0)},dash pattern=on 0.1pt off 0.8pt] (14.3,0) ellipse (0.3cm and 0.14cm);
\draw [rotate around={0.36:(16.71,-0.01)},dash pattern=on 0.1pt off 0.8pt] (16.71,-0.01) ellipse (0.3cm and 0.14cm);
\draw [shift={(4.5,2)}] plot[domain=0:pi,variable=\t]({1*1.5*cos(\t r)+0*1.5*sin(\t r)},{0*1.5*cos(\t r)+1*1.5*sin(\t r)});
\draw [shift={(4.51,1.99)}] plot[domain=-0.02:3.13,variable=\t]({1*0.91*cos(\t r)+0*0.91*sin(\t r)},{0*0.91*cos(\t r)+1*0.91*sin(\t r)});
\draw [rotate around={0.36:(3.3,2)},dash pattern=on 0.1pt off 0.8pt] (3.3,2) ellipse (0.3cm and 0.14cm);
\draw [rotate around={0.36:(5.71,1.99)},dash pattern=on 0.1pt off 0.8pt] (5.71,1.99) ellipse (0.3cm and 0.14cm);
\draw [shift={(8.5,2)}] plot[domain=0:pi,variable=\t]({1*1.5*cos(\t r)+0*1.5*sin(\t r)},{0*1.5*cos(\t r)+1*1.5*sin(\t r)});
\draw [shift={(8.51,1.99)}] plot[domain=-0.02:3.13,variable=\t]({1*0.91*cos(\t r)+0*0.91*sin(\t r)},{0*0.91*cos(\t r)+1*0.91*sin(\t r)});
\draw [rotate around={0.36:(7.3,2)},dash pattern=on 0.1pt off 0.8pt] (7.3,2) ellipse (0.3cm and 0.14cm);
\draw [rotate around={0.36:(9.71,1.99)},dash pattern=on 0.1pt off 0.8pt] (9.71,1.99) ellipse (0.3cm and 0.14cm);
\draw [shift={(12.5,2)}] plot[domain=0:pi,variable=\t]({1*1.5*cos(\t r)+0*1.5*sin(\t r)},{0*1.5*cos(\t r)+1*1.5*sin(\t r)});
\draw [shift={(12.51,1.99)}] plot[domain=-0.02:3.13,variable=\t]({1*0.91*cos(\t r)+0*0.91*sin(\t r)},{0*0.91*cos(\t r)+1*0.91*sin(\t r)});
\draw [rotate around={0.36:(11.3,2)},dash pattern=on 0.1pt off 0.8pt] (11.3,2) ellipse (0.3cm and 0.14cm);
\draw [rotate around={0.36:(13.71,1.99)},dash pattern=on 0.1pt off 0.8pt] (13.71,1.99) ellipse (0.3cm and 0.14cm);
\draw [shift={(16.5,2)}] plot[domain=0:pi,variable=\t]({1*1.5*cos(\t r)+0*1.5*sin(\t r)},{0*1.5*cos(\t r)+1*1.5*sin(\t r)});
\draw [shift={(16.51,1.99)}] plot[domain=-0.02:3.13,variable=\t]({1*0.91*cos(\t r)+0*0.91*sin(\t r)},{0*0.91*cos(\t r)+1*0.91*sin(\t r)});
\draw [rotate around={0.36:(15.3,2)},dash pattern=on 0.1pt off 0.8pt] (15.3,2) ellipse (0.3cm and 0.14cm);
\draw [rotate around={0.36:(17.71,1.99)},dash pattern=on 0.1pt off 0.8pt] (17.71,1.99) ellipse (0.3cm and 0.14cm);
\draw [line width=0.3pt,dotted] (1.5,2)-- (2.5,2);
\draw [line width=0.3pt,dotted] (0.5,0)-- (1.5,0);
\draw [line width=0.3pt,dotted] (-0.5,-2)-- (0.5,-2);
\draw [line width=0.3pt,dotted] (17.5,-2)-- (16.5,-2);
\draw [line width=0.3pt,dotted] (18.5,0)-- (17.5,0);
\draw [line width=0.3pt,dotted] (19.5,2)-- (18.5,2);
\end{tikzpicture}
}
\subfigure[Double Jacob's ladder]{
\begin{tikzpicture}[line cap=round,line join=round,>=triangle 45,x=1.0cm,y=1.0cm, scale=0.425]
\clip(-0.5,-1) rectangle (12.5,4); 
\draw (0,0)-- (12,0);
\draw (0,3)-- (12,3);
\draw (0,2.5)-- (0.5,2.5);
\draw [shift={(0.5,2)}] plot[domain=0:1.57,variable=\t]({1*0.5*cos(\t r)+0*0.5*sin(\t r)},{0*0.5*cos(\t r)+1*0.5*sin(\t r)});
\draw (1,2)-- (1,1);
\draw (1.5,1)-- (1.5,2);
\draw (2,2.5)-- (3,2.5);
\draw (2,0.5)-- (3,0.5);
\draw (3.5,2)-- (3.5,1);
\draw [shift={(2,2)}] plot[domain=1.57:pi,variable=\t]({1*0.5*cos(\t r)+0*0.5*sin(\t r)},{0*0.5*cos(\t r)+1*0.5*sin(\t r)});
\draw [shift={(2,1)}] plot[domain=pi:4.71,variable=\t]({1*0.5*cos(\t r)+0*0.5*sin(\t r)},{0*0.5*cos(\t r)+1*0.5*sin(\t r)});
\draw [shift={(3,1)}] plot[domain=-1.57:0,variable=\t]({1*0.5*cos(\t r)+0*0.5*sin(\t r)},{0*0.5*cos(\t r)+1*0.5*sin(\t r)});
\draw [shift={(3,2)}] plot[domain=0:1.57,variable=\t]({1*0.5*cos(\t r)+0*0.5*sin(\t r)},{0*0.5*cos(\t r)+1*0.5*sin(\t r)});
\draw (0,0.5)-- (0.5,0.5);
\draw [shift={(0.5,1)}] plot[domain=-1.57:0,variable=\t]({1*0.5*cos(\t r)+0*0.5*sin(\t r)},{0*0.5*cos(\t r)+1*0.5*sin(\t r)});
\draw (4,1)-- (4,2);
\draw (4.5,2.5)-- (5.5,2.5);
\draw (4.5,0.5)-- (5.5,0.5);
\draw (6,2)-- (6,1);
\draw [shift={(4.5,2)}] plot[domain=1.57:pi,variable=\t]({1*0.5*cos(\t r)+0*0.5*sin(\t r)},{0*0.5*cos(\t r)+1*0.5*sin(\t r)});
\draw [shift={(4.5,1)}] plot[domain=pi:4.71,variable=\t]({1*0.5*cos(\t r)+0*0.5*sin(\t r)},{0*0.5*cos(\t r)+1*0.5*sin(\t r)});
\draw [shift={(5.5,1)}] plot[domain=-1.57:0,variable=\t]({1*0.5*cos(\t r)+0*0.5*sin(\t r)},{0*0.5*cos(\t r)+1*0.5*sin(\t r)});
\draw [shift={(5.5,2)}] plot[domain=0:1.57,variable=\t]({1*0.5*cos(\t r)+0*0.5*sin(\t r)},{0*0.5*cos(\t r)+1*0.5*sin(\t r)});
\draw (6.5,1)-- (6.5,2);
\draw (7,2.5)-- (8,2.5);
\draw (7,0.5)-- (8,0.5);
\draw (8.5,2)-- (8.5,1);
\draw [shift={(7,2)}] plot[domain=1.57:pi,variable=\t]({1*0.5*cos(\t r)+0*0.5*sin(\t r)},{0*0.5*cos(\t r)+1*0.5*sin(\t r)});
\draw [shift={(7,1)}] plot[domain=pi:4.71,variable=\t]({1*0.5*cos(\t r)+0*0.5*sin(\t r)},{0*0.5*cos(\t r)+1*0.5*sin(\t r)});
\draw [shift={(8,1)}] plot[domain=-1.57:0,variable=\t]({1*0.5*cos(\t r)+0*0.5*sin(\t r)},{0*0.5*cos(\t r)+1*0.5*sin(\t r)});
\draw [shift={(8,2)}] plot[domain=0:1.57,variable=\t]({1*0.5*cos(\t r)+0*0.5*sin(\t r)},{0*0.5*cos(\t r)+1*0.5*sin(\t r)});
\draw (9,1)-- (9,2);
\draw (9.5,2.5)-- (10.5,2.5);
\draw (9.5,0.5)-- (10.5,0.5);
\draw (11,2)-- (11,1);
\draw [shift={(9.5,2)}] plot[domain=1.57:pi,variable=\t]({1*0.5*cos(\t r)+0*0.5*sin(\t r)},{0*0.5*cos(\t r)+1*0.5*sin(\t r)});
\draw [shift={(9.5,1)}] plot[domain=pi:4.71,variable=\t]({1*0.5*cos(\t r)+0*0.5*sin(\t r)},{0*0.5*cos(\t r)+1*0.5*sin(\t r)});
\draw [shift={(10.5,1)}] plot[domain=-1.57:0,variable=\t]({1*0.5*cos(\t r)+0*0.5*sin(\t r)},{0*0.5*cos(\t r)+1*0.5*sin(\t r)});
\draw [shift={(10.5,2)}] plot[domain=0:1.57,variable=\t]({1*0.5*cos(\t r)+0*0.5*sin(\t r)},{0*0.5*cos(\t r)+1*0.5*sin(\t r)});
\draw (11.5,2)-- (11.5,1);
\draw [shift={(12,2)}] plot[domain=1.57:pi,variable=\t]({1*0.5*cos(\t r)+0*0.5*sin(\t r)},{0*0.5*cos(\t r)+1*0.5*sin(\t r)});
\draw [shift={(12,1)}] plot[domain=pi:4.71,variable=\t]({1*0.5*cos(\t r)+0*0.5*sin(\t r)},{0*0.5*cos(\t r)+1*0.5*sin(\t r)});
\draw [rotate around={90:(0,2.75)}] (0,2.75) ellipse (0.25cm and 0.14cm);
\draw [rotate around={90:(0,0.25)}] (0,0.25) ellipse (0.25cm and 0.15cm);
\draw [rotate around={0:(1.25,1.5)},dash pattern=on 0.1pt off 0.8pt] (1.25,1.5) ellipse (0.26cm and 0.07cm);
\draw [rotate around={0:(3.75,1.5)},dash pattern=on 0.1pt off 0.8pt] (3.75,1.5) ellipse (0.26cm and 0.07cm);
\draw [rotate around={0:(6.25,1.5)},dash pattern=on 0.1pt off 0.8pt] (6.25,1.5) ellipse (0.26cm and 0.07cm);
\draw [rotate around={0:(8.75,1.5)},dash pattern=on 0.1pt off 0.8pt] (8.75,1.5) ellipse (0.26cm and 0.07cm);
\draw [rotate around={0:(11.25,1.5)},dash pattern=on 0.1pt off 0.8pt] (11.25,1.5) ellipse (0.26cm and 0.07cm);
\draw [rotate around={-87.89:(2.51,0.25)},dash pattern=on 0.1pt off 0.8pt] (2.51,0.25) ellipse (0.25cm and 0.12cm);
\draw [rotate around={-87.89:(5.04,0.25)},dash pattern=on 0.1pt off 0.8pt] (5.04,0.25) ellipse (0.25cm and 0.12cm);
\draw [rotate around={-87.89:(7.53,0.25)},dash pattern=on 0.1pt off 0.8pt] (7.53,0.25) ellipse (0.25cm and 0.12cm);
\draw [rotate around={-87.89:(10.02,0.25)},dash pattern=on 0.1pt off 0.8pt] (10.02,0.25) ellipse (0.25cm and 0.12cm);
\draw [rotate around={-87.89:(2.51,2.74)},dash pattern=on 0.1pt off 0.8pt] (2.51,2.74) ellipse (0.25cm and 0.12cm);
\draw [rotate around={-87.89:(5.05,2.75)},dash pattern=on 0.1pt off 0.8pt] (5.05,2.75) ellipse (0.25cm and 0.12cm);
\draw [rotate around={-87.89:(7.51,2.75)},dash pattern=on 0.1pt off 0.8pt] (7.51,2.75) ellipse (0.25cm and 0.12cm);
\draw [rotate around={-87.89:(10.04,2.75)},dash pattern=on 0.1pt off 0.8pt] (10.04,2.75) ellipse (0.25cm and 0.12cm);
\draw [rotate around={-87.89:(12,2.75)}] (12,2.75) ellipse (0.25cm and 0.12cm);
\draw [rotate around={-87.89:(12.01,0.25)}] (12.01,0.25) ellipse (0.25cm and 0.12cm);
\end{tikzpicture}
}
\caption{Leaves obtained by gluing handles to linear foliations}
\label{fig:leaves}
\end{figure}

\begin{example} \label{ex:surgery2}
Take a torus bundle $M_0$ over $\Sp^1$ with 
linear monodromy of type \eqref{nilmatrix}. This is a $\Sp^1$-bundle over $\T^2$ and its Euler class is $n$. We endow $M_0$ with a foliation $\F_0$ by dense cylinders using a  procedure due to Hector \cite{Hector}. To do so, we take
the product $\T^2 \times [0,1]$ foliated by the fibers of the map $p_2(x,y,t) = y$. 
We choose a diffeomorphim $\varphi : \T^2 \to \T^2$ sending meridians to meridians which projects to a diffeomorphism 
$\overline{\varphi}: \Sp^1 \to \Sp^1$.  If its rotation number is irrational, the quotient 
$M_\varphi = \T^2 \times [0,1] / (x,y,0) \sim (\varphi(x,y),1)$ inherits a minimal foliation by cylinders. Moreover, if $\varphi_\ast$ 
is given by \eqref{nilmatrix}, then $M_\varphi$ is homeomorphic to $M_0$. 
By Dehn surgery on a closed transversal $\gamma_0$, 
we have a foliated manifold $(M,\F)$ whose leaves have two nonplanar 
ends. 
Any parallel $p$ in the torus $T^2 \times \{0\}$ defines such a transversal $\gamma_0$ that turns exactly once in the $p$ direction. But this transversal is also induced by its image $\varphi_\ast(p) = p + nm$ that turns $n$ times in the 
$\Sp^1$ or $m$-direction.
If $n \neq 0$,
then $M$ is a graph-manifold by the same argument used in 
 Example~\ref{ex:surgery1}.
A similar discussion applies when $\gamma_0$ is any other closed transversal to $\F_0$.
\end{example}

\subsection{Gluing suspensions along linear foliations}
We propose another method to construct minimal foliations with leaves of negative curvature on graph-manifolds.
Since $PSL(2,\R)$ is a perfect group, any irrational rotation $R_\theta$ is a product of commutators, namely 
$$R_\theta=[f_1,h_1]\cdots[f_g, h_g]$$
for 
any $g \geq 1$. Let $\Sigma$ be the closed surface of genus $g$ and take $\Sigma_1 = \Sigma - \inte{D}$. 
The fundamental group $\Gamma_1=\pi_1(\Sigma_1)$  is a free group generated by classes $\alpha_i$ and $\beta_i$ with $1 \leq i \leq g$. Then one has a representation 
$$\rho_1: \Gamma_1 \to PSL(2,\R)\subset \rm{Diffeo}^\infty_+ (\Sp^1)$$
which sends $\alpha_i$ to $f_i$ and $\beta_i$ to $h_i$ and hence the product of commutators is sent to $R_\theta$. 
If we suspend $\rho_1$, we obtain a foliation
$\F_1$ on $M_1 = \Sigma_1 \times\Sp^1$ such that $\F_1|_{\partial M_1}$ is conjugated to the linear foliation $\F_\theta$ defined by $R_\theta$. 
The simple curve $\gamma_0$ representing the sum of $m$ and $p$ in $\pi_1(T^2)$ is transverse to  $\F_\theta$ and hence we  
can think of
$\F_\theta$  as the suspension of a diffeomorphism $f : \Gamma_0 \to\Gamma_0$ where $\Gamma_0$ is the image of $\gamma_0$. Since 
$\rm{Diffeo}^\infty_+(\Gamma_0) \cong \rm{Diffeo}^\infty_+(\Sp^1)$ is also perfect \cite[Corollaire 5.3]{Herman}, we can write $f$ as a product of commutators, namely 
$$f=[k_1,l_1]\cdots[k_n,l_n]$$
for any $n \geq 1$. 
Denoting by $\Gamma_2=\pi_1(\Sigma_2)$ the fundamental group of the surface of genus $n$ with the interior of a disk removed, 
there is a new representation 
$$\rho_2 : \Gamma_2 \to \rm{Diffeo}^\infty_+(\Gamma_0) \cong \rm{Diffeo}^\infty_+(\Sp^1)$$
which sends  $\alpha_i$ and $\beta_i$ (with $1 \leq i \leq n$) to $k_i$ and $l_i$ respectively.  If we suspend $\rho_2$,  we have a foliation $\F_2$ on $M_2 =\Sigma_2\times\Gamma_0$ which is again transverse to $\partial M_2$  and such that $\F_2|_{\partial M_2}$ is conjugate to the suspension of $f$. Finally, we 
can glue together $M_1$ and $M_2$ along the boundary by a $C^\infty$ diffeomorphism $\psi: \T^2 \to \T^2 $ which sends $\F_\theta$ to the 
suspension of $f$. 
Thus, we have a manifold $M$ endowed with a minimal foliation $\F$, but the gluing map doesn't 
preserve the Seifert fibrations on $\partial M_1$ and $\partial M_2$. 
Since the leaves of $\F$ are made of pieces, which are leaves of $\F_1$ and $\F_2$ having exponential growth with respect to the metrics induced on $M_1$ and $M_2$ by any Riemannian metric on $M$, we deduce:

\begin{proposition}
 Let $(M,\F)$ be a foliated 3-manifold which is obtained by gluing two suspensions along a linear foliation. Then $M$ is a graph-manifold and the leaves of $\F$ are hyperbolic surfaces. \qed
\end{proposition}

\subsection{Hirsch surgery} A third method to construct minimal foliations by hyperbolic surfaces on graph-manifolds is inspired by the classical Hirsch foliation. 
We start with two tori $\T^3$ each endowed with a linear foliation by cylinders 
transverse to the trivial fibration over $\T^2$,  and we construct two foliated mani\-folds $(M_1,\F_1)$ and 
$(M_2,\F_2)$ with boundary as follows. We choose a fiber as closed transversal 
$\gamma_1$, and we remove the interior of a closed neighborhood $V_1 \cong D \times \Sp^1$ to have a foliation by cylinders 
with holes on a manifold with boundary $M_1$ thus constructed. As in the cut-and-paste construction, the boundary $\partial M_1 \cong \Sp^1 \times \Sp^1$ is transverse to the linear foliation and we can assume that the induced foliation is conjugate 
to the  horizontal foliation by circles $\Sp^1 \times \{ z \}$. 
Next, we choose a closed transversal $\gamma_2$ that 
turns two times in the fiber direction, and similarly we obtain another foliation by cylinders with holes on another manifold with boundary 
$M_2$. 
\medskip 

Finally, we take the closed unit disk $D_1 \subset \C$ minus the interior of two disks $D_2^\pm$ of radius $1/4$ 
centered at $\pm 1/2$, the product of a pair of pants $P = D_1 - \inte{D}_2^- \cup \inte{D}_2^+$ and $\Sp^1$, and the quotient 
$M_0 = P \times \Sp^1 / (Z,z) \sim (-Z,-z)$ with the foliation induced by the horizontal foliation of
$P \times \Sp^1$.
We attach $M_0$ to $M_1 \sqcup M_2$ via a diffeomorphism $\psi$ which sends the image  of
$\partial D_1 \times \Sp^1$ in $M_0$ onto $\partial M_1$ and the image of 
$\partial \inte{D}_2^-  \times \Sp^1\sqcup  \inte{D}_2^+  \times \Sp^1$ in $M_0$ onto $\partial M_2$ preserving the horizontal 
foliations on each torus. We denote by $\partial_{out} M_0$ and $\partial_{in} M_0$ these two components of $\partial M_0$. Thus, we have a minimal foliation $\F$ on a 
manifold $M$ which splits naturally into 
three 
pieces: 
$M_0$ and $M_1$ are Seifert fibered (with one exceptional fiber and without exceptional fibers respectively), but $M_2$ admits a Seifert fibered structure, a toral decomposition or a hyperbolic structure depending on the isotopy class of $\gamma_2$ (as in Example~\ref{ex:surgery1}.(ii)).  
As the gluing map $\psi$ does not preserve the Seifert fibered structures induced on $\partial_{out} M_0$ and $\partial M_1$, the toral decomposition of $M$ is always nontrivial, containing or not a hyperbolic piece. If $M_2$ is Seifert fibered and the Seifert fibered structure induced on $\partial M_2$ is isomorphic to that of $\partial_{in} M_0$, then $M$ is graph-manifold. We say that $(M,\F)$ is given by \emph{Hirsch surgery} from two linear foliations on $\T^3$.

\begin{proposition} \label{prop:Hirsch}
Let $(M,\F)$ be a foliated 3-manifold given by Hirsch surgery from two linear foliations on $\T^3$. 
 There is a choice of the gluing map $\psi$ such the all the leaves of $\F$ are geometrically infinite hyperbolic surfaces with one or two ends and trivial holonomy.
\end{proposition}
\begin{proof}
Take the singular fiber of $M_0$ as closed transversal $\gamma_0$ which meets all leaves, and identify it with $\Sp^1$ 
 (see Figure~\ref{fig:Hirsch}). We will describe the holonomy pseudogroup $\Gamma$ of the foliation $\F$ reduced to $\gamma_0$. To do so, as the meridians in $\partial M_i$, $\partial_{out}M_0$ and $\partial_{in}M_0$ have no holonomy, we can assume without loss of generality that $\psi$ preserves the Seifert fibered structures of these tori. We also assume that the linear foliations that constitute the starting point of the construction are given by two irrational numbers $\alpha$ and $\beta$. Then the pseudogroup reduced to any fiber in $\partial M_1$ is related to $\Gamma$ by a pseudogroup equivalence $\Phi_1$ (in the sense of \cite{Haefliger2}) defining a diffeomorphism (which is still denoted by $\Phi_1$) from this fiber onto $\gamma_0$. Then $\Phi_1$ conjugates the rotation 
$R_\alpha(z) = e^{2 \pi i \alpha}z$ acting on each point $z$ of the fiber to a diffeomorphism $\Phi_1 \scirc R_\alpha \scirc \Phi_1^{-1}$ acting on $\gamma_0$.
Similarly, the pseudogroup reduced to any fiber in $M_2$ is also related to $\Gamma$ by another pseudogroup equivalence $\Phi_2$ defining a new diffeomorphism (which is still denoted by $\Phi_2$ as before) from the fiber to $\gamma_0$. In this case, $\Phi_2$ is constructed by passing through a copy of $\gamma_2$ in $\partial M_2$ that projects onto $\gamma_0$ by the doubling map $f(z)=z^2$. Now $\Phi_2$ conjugates the rotation  $R_\beta(z) = e^{2 \pi i \beta}z$ acting on each point $z$ of the fiber to another diffeomorphism $\Phi_2 \scirc R_\beta \scirc \Phi_2^{-1}$ acting on $\gamma_0$. If the gluing map $\psi$ retricts to the identity for each parallel in $\partial_{out}M_0$ and $\partial_{in}M_0$, then both parametrizations of $\gamma_0$ coincide and the holonomy pseudogroup is generated by the rotations $R_\alpha$ and $R_\beta$ up to conjugation. Thus, there is a holonomy representation of $\pi_1(M)$ into the group of rotations generated by $R_\alpha$ and $R_\beta$ so that the foliation $\F$ is without holonomy. The leaves of $\F$ have one or two ends depending on whether $\alpha$ and $\beta$ are linearly independent over $\Z$ or not. Moreover, all leaves of $\F$ are geometrically infinite hyperbolic surfaces as they contain infinitely many pairs of pants approaching  all the ends.
\end{proof}

 A similar conclusion holds if we take linear foliations by planes instead of  linear foliations by cylinders.

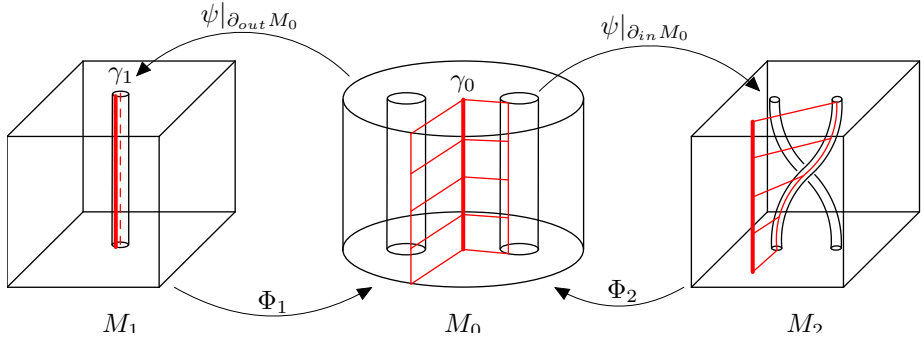
\begin{figure}
\definecolor{ffqqqq}{rgb}{1.,0.,0.}
\begin{tikzpicture}[line cap=round,line join=round,>=triangle 45,x=1.0cm,y=1.0cm, scale=1]
\clip(-6,-1.6) rectangle (6.5,2.8);
\draw [line width=0.5pt] (-3.,0.)-- (-5.,0.);
\draw [line width=0.5pt] (-5.,2.)-- (-5.,0.);
\draw [line width=0.5pt] (-5.,0.)-- (-6.,-1.);
\draw [line width=0.5pt] (-4.,-1.)-- (-6.,-1.);
\draw [line width=0.5pt] (-6.,-1.)-- (-6.,1.);
\draw [line width=0.5pt] (-5.,2.)-- (-3.,2.);
\draw [line width=0.5pt] (-4.,1.)-- (-6.,1.);
\draw [line width=0.5pt] (-6.,1.)-- (-5.,2.);
\draw [line width=0.5pt] (-4.,1.)-- (-3.,2.);
\draw [line width=0.5pt] (-3.,2.)-- (-3.,0.);
\draw [line width=0.5pt] (-3.,0.)-- (-4.,-1.);
\draw [line width=0.5pt] (-4.,-1.)-- (-4.,1.);
\draw [line width=0.5pt] (4.,2.)-- (4.,0.);
\draw [line width=0.5pt] (4.,0.)-- (3.,-1.);
\draw [line width=0.5pt] (5.,-1.)-- (3.,-1.);
\draw [line width=0.5pt] (3.,-1.)-- (3.,1.);
\draw [line width=0.5pt] (4.,2.)-- (6.,2.);
\draw [line width=0.5pt] (5.,1.)-- (3.,1.);
\draw [line width=0.5pt] (3.,1.)-- (4.,2.);
\draw [line width=0.5pt] (5.,1.)-- (6.,2.);
\draw [line width=0.5pt] (6.,2.)-- (6.,0.);
\draw [line width=0.5pt] (6.,0.)-- (5.,-1.);
\draw [line width=0.5pt] (5.,-1.)-- (5.,1.);
\draw [rotate around={-0.042844339263239:(-0.0019251997388241016,-0.49887689799586593)},line width=0.5pt] (-0.0019251997388241016,-0.49887689799586593) ellipse (1.582611848474999cm and 0.49887843806365556cm);
\draw [rotate around={0.:(0.,1.5)},line width=0.5pt] (0.,1.5) ellipse (1.581138830084192cm and 0.5cm);
\draw [line width=0.5pt] (-1.5820226625840632,1.4892868684980738)-- (-1.5820226625840632,-0.5151078629305075);
\draw [line width=0.5pt] (1.582428379358689,1.5124100442968271)-- (1.582428379358689,-0.49198468713175403);
\draw [rotate around={-0.009665691357372922:(-0.7490983572880869,-0.4950221683821027)},line width=0.5pt] (-0.7490983572880869,-0.4950221683821027) ellipse (0.25118235155051544cm and 0.07464120329745294cm);
\draw [line width=0.5pt] (-1.,1.5)-- (-1.0045368409477586,-0.4933758119239331);
\draw [line width=0.5pt] (-0.49606010928037214,1.5)-- (-0.49606010928037214,-0.5);
\draw [line width=0.5pt] (0.48835316642922055,1.4991137393879372)-- (0.4838163254814619,-0.49426207253599563);
\draw [line width=0.5pt] (0.988702337467057,1.4933758119239329)-- (0.9841654965192983,-0.5);
\draw [rotate around={-0.009665691357353031:(-0.746599788283999,1.5049594624231368)},line width=0.5pt] (-0.746599788283999,1.5049594624231368) ellipse (0.251182351550513cm and 0.07464120329745219cm);
\draw [rotate around={-0.009665691357379554:(0.7386015439000909,-0.4944105023872424)},line width=0.5pt] (0.7386015439000909,-0.4944105023872424) ellipse (0.25118235155051916cm and 0.07464120329745404cm);
\draw [rotate around={-0.009665691357353044:(0.7382469077997075,1.5057717286760155)},line width=0.5pt] (0.7382469077997075,1.5057717286760155) ellipse (0.25118235155052276cm and 0.07464120329745523cm);
\draw [line width=0.5pt] (-4.6149185388602225,1.5544282573636903)-- (-4.619455379807982,-0.4389475545602426);
\draw [line width=0.5pt] (-4.399949721514286,1.5618230712685477)-- (-4.399949721514286,-0.4381769287314523);
\draw [rotate around={0.:(-4.510827174886065,-0.43662457643929264)},line width=0.5pt] (-4.510827174886065,-0.43662457643929264) ellipse (0.10637260497526571cm and 0.03626473616660568cm);
\draw [rotate around={0.:(-4.506040690025405,1.5637675415341568)},line width=0.5pt] (-4.506040690025405,1.5637675415341568) ellipse (0.1063726049754517cm and 0.036264736166669094cm);
\draw [shift={(5.266966065865738,1.4157052770808942)},line width=0.5pt]  plot[domain=3.0805886200929447:3.92,variable=\t]({1.*1.2369388370916523*cos(\t r)+0.*1.2369388370916523*sin(\t r)},{0.*1.2369388370916523*cos(\t r)+1.*1.2369388370916523*sin(\t r)});
\draw [shift={(3.679070839246721,-0.38221168350492823)},line width=0.5pt]  plot[domain=-0.08484095039762352:0.74,variable=\t]({1.*1.1618164805087425*cos(\t r)+0.*1.1618164805087425*sin(\t r)},{0.*1.1618164805087425*cos(\t r)+1.*1.1618164805087425*sin(\t r)});
\draw [shift={(5.326808782185721,1.4422959154791912)},line width=0.5pt]  plot[domain=3.105815575594726:3.9,variable=\t]({1.*1.169136709651012*cos(\t r)+0.*1.169136709651012*sin(\t r)},{0.*1.169136709651012*cos(\t r)+1.*1.169136709651012*sin(\t r)});
\draw [shift={(3.636032100204356,-0.4074590108421598)},line width=0.5pt]  plot[domain=-0.054697044718346355:0.74,variable=\t]({1.*1.3370392652215342*cos(\t r)+0.*1.3370392652215342*sin(\t r)},{0.*1.3370392652215342*cos(\t r)+1.*1.3370392652215342*sin(\t r)});
\draw [rotate around={-2.9579506036451946:(4.094718111937447,1.4888116346403715)},line width=0.5pt] (4.094718111937447,1.4888116346403715) ellipse (0.06408449371547918cm and 0.037860143942234095cm);
\draw [rotate around={1.708087982870742:(4.90511095401719,-0.47936261918567674)},line width=0.5pt] (4.90511095401719,-0.47936261918567674) ellipse (0.06783740520897243cm and 0.02763841492970122cm);
\draw [shift={(3.7440959891077994,1.4096076375420026)},line width=0.5pt]  plot[domain=-0.8350950233349614:0.06901207967980803,variable=\t]({1.*1.2369388370916508*cos(\t r)+0.*1.2369388370916508*sin(\t r)},{0.*1.2369388370916508*cos(\t r)+1.*1.2369388370916508*sin(\t r)});
\draw [shift={(5.346337949256641,-0.3755358718207614)},line width=0.5pt]  plot[domain=2.297731783855797:3.2344416501703748,variable=\t]({1.*1.1618164805087419*cos(\t r)+0.*1.1618164805087419*sin(\t r)},{0.*1.1618164805087419*cos(\t r)+1.*1.1618164805087419*sin(\t r)});
\draw [shift={(3.6840422548140546,1.435718205217922)},line width=0.5pt]  plot[domain=-0.8327054016449811:0.043785124178027705,variable=\t]({1.*1.169136709651006*cos(\t r)+0.*1.169136709651006*sin(\t r)},{0.*1.169136709651006*cos(\t r)+1.*1.169136709651006*sin(\t r)});
\draw [shift={(5.389577487897337,-0.40043773709568486)},line width=0.5pt]  plot[domain=2.3284619159106854:3.2042977444910963,variable=\t]({1.*1.3370392652215386*cos(\t r)+0.*1.3370392652215386*sin(\t r)},{0.*1.3370392652215386*cos(\t r)+1.*1.3370392652215386*sin(\t r)});
\draw [rotate around={-176.58322214792545:(4.915720922977141,1.4920989664216837)},line width=0.5pt] (4.915720922977141,1.4920989664216837) ellipse (0.06408449371547918cm and 0.037860143942234095cm);
\draw [rotate around={178.75073926555865:(4.121115127383414,-0.4825017733550097)},line width=0.5pt] (4.121115127383414,-0.4825017733550097) ellipse (0.06783740520897243cm and 0.02763841492970122cm);
\draw [line width=0.5pt] (4.,0.)-- (6.,0.);
\draw [line width=1.5pt,color=ffqqqq] (-6.186585563995027E-4,-0.49856798064849217)-- (1.2818962243619218E-4,1.4887250269450296);
\draw [line width=1.5pt,color=ffqqqq] (-4.5725788967603895,1.5354735035282983)-- (-4.570034880658589,-0.4667524534181446);
\draw [line width=0.5pt,dash pattern=on 3pt off 3pt,color=ffqqqq] (-4.508575508150605,1.5728808678671051)-- (-4.509936983651658,-0.4430990995696002);
\draw [line width=0.5pt,,color=ffqqqq] (-0.6884973024400652,-0.9585625659709638)-- (-0.6884973024400652,1.0351068880690035);
\draw [line width=0.5pt,color=ffqqqq] (-0.6884973024400652,1.0351068880690035)-- (1.2818962243619218E-4,1.4887250269450296);
\draw [line width=0.5pt,color=ffqqqq] (-0.6884973024400652,-0.9585625659709638)-- (-6.186585563995027E-4,-0.49856798064849217);
\draw [line width=0.5pt,color=ffqqqq] (-0.6880876955811462,0.5045920934351044)-- (5.377964813551017E-4,0.9582102323111304);
\draw [line width=0.5pt,color=ffqqqq] (-0.685319377043518,0.00957699299930953)-- (0.00330611501898323,0.46319513187533584);
\draw [line width=0.5pt,color=ffqqqq] (-0.6928397768953904,-0.4889112700163013)-- (-0.004214284832889353,-0.03529313114027503);
\draw [shift={(5.288688292550559,-0.36240933174139967)},line width=0.5pt,color=ffqqqq]  plot[domain=2.3028357334584393:3.268087831296908,variable=\t]({1.*1.1710659891225945*cos(\t r)+0.*1.1710659891225945*sin(\t r)},{0.*1.1710659891225945*cos(\t r)+1.*1.1710659891225945*sin(\t r)});
\draw [shift={(3.7430403922623356,1.395296202989313)},line width=0.5pt,color=ffqqqq]  plot[domain=-0.8602653328762075:0.050305659439837176,variable=\t]({1.*1.1697009446703166*cos(\t r)+0.*1.1697009446703166*sin(\t r)},{0.*1.1697009446703166*cos(\t r)+1.*1.1697009446703166*sin(\t r)});
\draw [line width=1.5pt,color=ffqqqq] (3.8063889550975976,-0.798024071769694)-- (3.8063889550975976,1.1968074066478707);
\draw [line width=0.5pt,color=ffqqqq] (3.8063889550975976,-0.798024071769694)-- (4.126978948112674,-0.5101487981184611);
\draw [line width=0.5pt,color=ffqqqq] (3.8063889550975976,1.1968074066478707)-- (4.911261591704226,1.4541139650785342);
\draw [line width=0.5pt,color=ffqqqq] (3.8063889550975976,-0.2937283770175677)-- (4.156358352077328,-0.06370470039536674);
\draw [line width=0.5pt,color=ffqqqq] (3.8063889550975976,0.19570816361771315)-- (4.47372870738174,0.47856416690237213);
\draw [line width=0.5pt,color=ffqqqq] (3.8063889550975976,0.7148075248975565)-- (4.833697864923527,0.9726065091346674);
\draw [line width=0.5pt,color=ffqqqq] (0.5930432685649358,1.4448883553619918)-- (0.5928615780505196,-0.5551805706217298);
\draw [line width=0.5pt,color=ffqqqq] (1.2818962243619218E-4,1.4887250269450296)-- (0.5930432685649358,1.4448883553619918);
\draw [line width=0.5pt,color=ffqqqq] (5.377964813551017E-4,0.9582102323111304)-- (0.5929968906434262,0.9343551557552815);
\draw [line width=0.5pt,color=ffqqqq] (0.00330611501898323,0.46319513187533584)-- (0.592950523435148,0.423939888563164);
\draw [line width=0.5pt,color=ffqqqq] (-6.186585563995027E-4,-0.49856798064849217)-- (0.5928615780505196,-0.5551805706217298);
\draw [line width=0.5pt,color=ffqqqq] (-0.004214284832889353,-0.03529313114027503)-- (0.5929048382512246,-0.07896758041793639);
\node (1) at (-4.5,1.8) {$\gamma_1$};
\node (2) at (0,1.7) {$\gamma_0$};
\node(A) at (-4.5,-1.5) {$M_1$};
\node(B) at (0,-1.5) {$M_0$};
\node(C) at (4.5,-1.5) {$M_2$};
\node(3) at (-2.8,2.6) {$\psi|_{\partial_{out}M_0}$};
\draw [->, line width=0.2pt] (-1.5,1.8) to [out=135,in=45] (-4.3,1.8);
\node(4) at (2.4,2.45) {$\psi|_{\partial_{in}M_0}$};
\draw [->,  line width=0.2pt] (1,1.6) to [out=45,in=135] (3.9,1.6);
\draw [->, line width=0.2pt] (-3.8,-1) to [out=-30,in=-150]   (-1.2,-1);
\node(5) at (-2.5,-1.2) {$\Phi_1$};
\draw [->,line width=0.2pt] (2.9,-1) to [out=-150,in=-30]  (1.2,-1)  ;
\node(5) at (2.1,-1.05) {$\Phi_2$};

\end{tikzpicture}
\caption{Gluing pieces for Hirsch surgery}
\label{fig:Hirsch}
\end{figure}

\section{Foliations without holonomy on hyperbolic 3-manifolds} \label{Swithoutholonomy}

 From now on, we deal with the construction of minimal foliations on hyperbolic 3-manifolds. Under a homological condition that will be made precise next, it is actually possible to construct foliations without holonomy proving Theorem~\ref{thm:withoutholonomy}.

\subsection{A Tischler type construction}
Let $M$ be a fiber bundle over $\Sp^1$ with pseudo-Anosov monodromy, and assume that its first Betti number $b_1 \geq 2$.  
From \cite{Katok&al}, there is a $C^\infty$ diffeomorphism $\varphi : \Sigma \to \Sigma$, called {\em of pseudo-Anosov type}, such that  $M$ is homeomorphic to $M_\varphi$ by a homeomorphism preserving the fibrations of $M$ and $M_\varphi$ over $\Sp^1$. Denote by $f : M_\varphi \to \Sp^1$ the bundle map.
Then the canonical 1-form $d\theta$ on $\Sp^1$ lifts to a closed $1$-form  
$\omega_0=f^*(d\theta)$ on $M_\varphi$. Let $\omega$ be a nonvanishing closed $1$-form on $M_\varphi$ 
which is $\C^\infty$ close to $\omega_0$ and such that its group of periods $Per( \omega) \cong
\Z^r$ for some $2 \leq r \leq b_1$.
Recall that $Per( \omega)$ is the additive subgroup of $\R$ generated by 
the set of integrals $\int_{\gamma_i}\, \omega$ where 
$\{\gamma_1, \dots, \gamma_{b_1}\}$
are loops representing a basis of $H_1(M_\varphi, \R)$. The nonvanishing $1$-form $\omega$ defines a foliation without 
holonomy $\F$ whose leaves are diffeomorphic to the regular covering of $\Sigma$ with deck 
transformation group $\Z^{r-1}$ (see \cite[\S VIII.1.1]{Hector-Hirsch}). Since $\F$ has no holonomy
and its leaves are non-compact, Sacksteder's theorem  implies that it is minimal (see 
 \cite[Theorem VI.3.2]{Hector-Hirsch}).

\subsection{Action on homology of a pseudo-Anosov map}
Now, we shall see how
all manifolds in the hypothesis of Theorem~\ref{thm:withoutholonomy}  can be obtained.
Given the canonical basis $\{[\alpha_1], [\beta_1], \dots, [\alpha_g], [\beta_g]\}$,
the algebraic intersection number defines a symplectic structure on the homology group $H_1(\Sigma,\R) \cong \R^{2g}$.
Each orientation-preserving homeomorphism $\varphi : \Sigma \to \Sigma$ induces an automorphism $\varphi_\ast : H_1(\Sigma,\R) \to H_1(\Sigma,\R)$ respecting this structure 
and hence the mapping class group $\rm{Mod}(\Sigma)$ admits a \emph{homology representation} 
$\Psi : \rm{Mod}(\Sigma) \to Sp(2g,\Z)$ on the symplectic group 
$ Sp(2g,\Z)$.
 The \emph{Torelli group} $\mathcal{T}(\Sigma)$ is the subgroup of $\rm{Mod}(\Sigma)$ consisting of those elements which act trivially on $H_1(\Sigma,\R)$, 
 that is, $\mathcal{T}(\Sigma) = Ker \,\Psi$. 

 \begin{definition} \label{def:Torelli} For each 
element 
$\varphi \in \rm{Mod}(\Sigma)$, we denote by $k(\varphi)$ the dimension of the 
 invariant subspace of $H_1(\Sigma, \R)$ which is fixed by $\varphi_\ast$. We say $\varphi$ is \emph{partially Torelli of order $k(\varphi)$}. In particular, if $\varphi \in \mathcal{T}(\Sigma)$, then $k(\varphi) = 2g$. 
\end{definition} 

 By Wang's homology sequence, the first homology group $H_1(M_\varphi, \R) \cong \R^{k(\varphi)+1}$ and the first Betti number $b_1 = k(\varphi)+1$. Examples of pseudo-Anosov maps in the Torelli group were constructed by Thurston in \cite{Thurston1}. In fact,
for any $0\leq k\leq 2g$, there exist pseudo-Anosov 
maps which are partially Torelli of order $k$. Explicit examples have been constructed 
in \cite{Agol} when $k>0$ and 
in \cite{Penner} when $k=0$. 
This proves that any hyperbolic surface bundle over $\Sp^1$ having first Betti number $b_1 \geq 2$ actually has a foliation without holonomy by geometrically infinite surfaces of polynomial growth.

\setcounter{equation}{0} 

\section{Foliations via branched coverings} \label{Sbranching}

In this section, we are interested in a very general method for adding topological complexity to the ambient manifold and the leaves of a foliation, which is valid in any dimension. Let $\pi : M \to M_0$ be a branched covering between closed 3-manifolds with
branch locus $B_0  \subset M_0$ and ramification locus $B \subset M$ of dimension one. In fact, we assume that the branch locus $B_0$ has no $0$-dimensional stratum, that is, $B_0$ is a union of simple closed curves. 
We also assume that $\F_0$ is a foliation by surfaces of $M_0$ transverse to $B_0$. Then $\F_0$  lifts to a foliation by surfaces $\F$ of $M$ that coincides with the usual lifted foliation $\pi^\ast \F$ on the complement of $B= \pi^{-1}(B_0)$. Indeed, if $\F_0$ is given by a $1$-form $\omega_0$ on $M_0$ satisfying the (closed) integrability condition $\omega_0 \wedge d\omega_0 = 0$, we can modify the differentiable structure of $M$ along a tubular neighborhood of $B$ to match fibers (which are branched covers of the fibers of a tubular neighborhood of $B_0$) with disks tangent to $\F$. Then $\omega_0$ can be lifted to a $1$-form $\omega$ which has no singularities and defines $\F$ on $M-B$. By continuity, the integrability condition $\omega \wedge d\omega = 0$ holds on the whole manifold $M$. Next we apply this method to construct minimal foliations by geometrically infinite surfaces on a large family of hyperbolic 3-manifolds that fiber over $\Sp^1$.

\subsection{Minimal foliations on hyperbolic 3-manifolds}
 Let  $\varphi : \Sigma \to \Sigma$ 
be an orientation-preserving pseudo-Anosov homeomorphism,
and assume the measured foliations $\F_\varphi^+$ are orientable.

\begin{definition} 
We say $\varphi$ is {\em of Franks-Rykken type} if there are a hyperbolic toral automorphism $A : \T^2 \to \T^2$ and a branched 
covering $\pi_0 : \Sigma \to \T^2$ such that $\pi_0$ is a  semiconjugacy from $\varphi$ to $A$ that makes the following diagram commute:

\begin{equation} \label{eq:Franks1}
\begin{tikzcd}        
    \Sigma \arrow{r}{\varphi} \arrow[swap]{d}{\pi_0} & \Sigma \arrow{d}{\pi_0} \\
    \, \T^2 \arrow{r}{A} & \,\T^2
\end{tikzcd}
\end{equation} 
\end{definition}

The branched covering $\pi_0 : \Sigma \to \T^2$ extends to a branched covering $\pi : M_\varphi \to \T^3_A$ such that $f_0 \scirc \pi = f$ where $f : M_\varphi \to \Sp^1$ and $f_0 : \T^3_A \to \Sp^1$ are the fibrations of the mapping tori of $\varphi$ and  $A$ respectively.
In \cite{Franks}, 
Franks and 
Rykken give a necessary and sufficient condition to have monodromy of this type: the stretch factor of $\varphi$  must be quadratic over $\mathbb{Q}$. Consequently, we can reformulate Theorem~\ref{thm:Franks} as follows:

\begin{theorem} \label{thm:Franks2}
 Let $M = M_\varphi$ be a closed orientable hyperbolic 3-manifold that fibers over the circle $\Sp^1$. Asume that the monodromy $\varphi$ is of Franks-Rykken type.
Then $M$ admits a minimal transversely affine foliation $\F$ whose leaves are  geometrically 
 infinite hyperbolic surfaces.
\end{theorem}

\begin{proof} Firstly, up to conjugation by a homeomorphism isotopic to the identity, we can  assume that $\varphi$ is a $C^\infty$ diffeomorphism of both pseudo-Anosov and  Franks-Rykken type. Then the  branched covering $\pi_0 : \Sigma \to \T^2$ in Diagram \eqref{eq:Franks1} extends to a branched covering 
$\pi : M \to \T^3_A$ from the hyperbolic 3-manifold $M$ to the solvable 3-manifold $\T^3_A$. 
We shall apply the announced method to the center-unstable foliation $\F_0$ of the Anosov flow (see Section \ref{Sintro}). 
By construction, the branch locus $B_0$ of $\pi$ is a link in $\T^3_A$ made up of periodic orbits $\O_1, \dots, \O_n$ of the Anosov flow and its ramification locus $B$ is a link in $M$ made up of periodic orbits of the pseudo-Anosov flow whose component project diffeomorphically onto the orbits $\O_j$. 
But replacing the orbits $\O_1, \dots, \O_n$ by isotopic closed transversals $\gamma_1, \dots, \gamma_n$ to both the fibration $f_0$ and the foliation $\mathcal F_0$, 
the branch locus 
 $B_0$ becomes transverse to both $f_0$ and $\F_0$. We lift $\F_0$ to a foliation by surfaces $\F$ on $M$ such that $B$ is transverse to both the fibration $f$ and the foliation $\F$. By minimality, each component $\gamma_j$ of $B_0$ is a complete 
transversal to $\F_0$ (i.e. it meets all leaves) and therefore each component of $B$ is also a complete transversal to $\F$. The local description of the branched covering around this complete transversal tells us that the holonomy pseudogroups of the foliations $\F_0$ and $\F$ are conjugate. Therefore, $\F_0$ and $\F$ share the same transverse structure
and hence $\F$ is minimal and transversely affine.

On the other hand, each leaf $L \in \F$ is a branched cover of some leaf $L_0 \in \F_0$. Let $\pi_L : L \to L_0$ be the branched covering we are considering, and let $d$ be the covering degree of $\pi_L$. It is smaller than or equal to the degree of $\pi$ and may depend on $L$.
By minimality, $L_0$ contains countably many branch points in each component of $B_0$, and each end of $L_0$ is approached by these branch points. The leaf $L$ also contains countably many ramification points with the same index in each component of $B$, and each end of $L$ is also approached by these ramification points. 
We cover $L_0$ by an increasing sequence of closed disk or cylinders  $D_k$ (depending on whether $L_0$ is a plane or a cylinder) that contain a 
finite number of branch points $p_1, \dots, p_k$. Then $L$ can be covered by an increasing sequence of branched covers $D'_k = \pi_L^{-1}(D_k)$ with degre $d$ and ramifications points $q_1, \dots , q_l$ such that $q_j$ has ramification index $2 \leq e_j \leq d$. 
Now, the Riemann-Hurwitz formula implies that
$\chi(D'_k) = d \chi(D_k) - \sum_{j=1}^l (e_j - 1) \leq d - l$
is not bounded from below.
Consequently, all leaves of $\F$ have infinite genus. Indeed, if $L_0$ is a
planar leaf (resp. cylindrical leaf), then $L$ admits at most $d$ ends (resp. $2d$ ends). Let
us show this for a planar leaf $L_0$. If the disk $D_k \subset L_0$ is sufficiently large, then
$D'_k = \pi_L^{-1}(D_k)$ is connected. Given a non-branched  base point $ x_0 \in D_k$, consider
$\pi_L^{-1}(x_0) = \{  z_0, \dots , z_{d-1} \}$ and take a path $\sigma_j$ in $L$ which joins $z_0$ to $z_j$ and evades the ramification points for $j=1, \dots , d-1$. If $D_k$ is large enough so as to include all the
images $\pi_L(\sigma_j)$, then the component of $D'_k$ which contains $z_0$ contains all $z_j $, by
the path lifting property. Therefore $D'_k$ is connected. Then
the number of connected components of $ L - \inte{D}'_k$ is at most $d$. Since $\{D'_k\}$ is 
an exhausting sequence of $L$, we conclude that $L$ has at most $d$ ends. Since $\chi(D'_k)$ converges to $-\infty$ as $k$ goes to $+\infty$, the leaf $L$ has infinite genus. In fact, each end is nonplanar.
\end{proof}

\subsection{Action on homology of pseudo-Anosov maps of Franks-Rykken type} As in Section~\ref{Swithoutholonomy}, it is natural to ask about the action on homology of a pseudo-Anosov map $\varphi : \Sigma \to \Sigma$ making commutative Diagram~\eqref{eq:Franks1}. This induces another diagram
\begin{equation} \label{eq:Franks3}
\begin{tikzcd}        
    H_1(\Sigma, \R) \arrow{r}{\varphi_\ast} \arrow[swap]{d}{(\pi_0)_\ast} &  H_1(\Sigma, \R)  \arrow{d}{(\pi_0)_\ast} \\
    \,  H_1(\T^2,\R) \arrow{r}{A_\ast} & \, H_1(\T^2,\R)
\end{tikzcd}
\end{equation} 
in homology. Since $A_\ast$ fixes no nontrivial subspace of $H_1(\T^2,\R)$, we have that the subspace of $H_1(\Sigma, \R)$ fixed by $\varphi_\ast$ is contained in the kernel of $(\pi_0)_\ast$.  Therefore,  its dimension $k(\varphi) \leq 2g-2$ and hence $M_\varphi$ has first Betti number 
$b_1 \leq 2g-1$. 
In Section~\ref{SAnosov}, we will construct two examples 
from pseudo-Anosov maps with stretch factor quadratic over $\mathbb{Q}$ and Torelli order $k=0$, but we do not know examples with $1 \leq  k \leq 2g-2$. 

\section{Two examples} \label{SAnosov}

Here, we illustrate Theorem~\ref{thm:Franks} by describing two specific examples: 

\subsection{An even number of branch points}
Let  $\T^3_A$ be the solvable 3-manifold obtained from the suspension of Arnold's cat map
\begin{equation}\label{Arnold}
A = \matriz{2}{1}{1}{1}.
\end{equation}
We still denote by $\F_0$ the center-unstable foliation of the Anosov flow, and
we exemplify Theorem~\ref{thm:Franks} by constructing directly a branched cover of $\T^3_A$ endowed with a foliation by hyperbolic surfaces obtained from $\F_0$.
For any even number $n = 2g -2$ with $g \geq 2$, we choose $n$ periodic points  $p_1, \dots , p_n$ of 
$A$ and the corresponding periodic orbits $\O_1, \dots, \O_n$ of the Anosov flow. Up to replacing $A$ by one of its powers $A^m$, we can actually assume that 
$p_1, \dots , p_n$ are fixed by $A^m$. In fact, we assume it has one more fixed point $x_0$.
As explained by Laudenbach in \cite[Expos\'e 13]{Fathi&al}, there is a closed surface $\Sigma$ of genus $g$ which is the total space of a 2-fold branched covering of the torus $\T^2$ 
with branch points $p_1, \dots p_n$. 
 Indeed, if  $m$ and $p$ are the generators of $\pi_1(\T^2)$ and $\alpha_1, \dots,  \alpha_n$ are the classes represented by simple loops which surround $p_1, \dots , p_n$ counterclockwise, the fundamental group $\pi_1(\T^2 - \{p_1, \dots , p_n\}, x_0)$ is a free group generated by $m$, $p$ and $\alpha_2, \dots, \alpha_n$, or equivalently given by 
$$
\pi_1(\T^2 - \{p_1, \dots , p_n\}, x_0) = | \; m,p, \alpha_1, \dots, \alpha_n : [m,p] \prod_{j=1}^n \alpha_j = 1 \; |.
$$
As $n$ is even, there is a representation 
\begin{equation} \label{branchrep}
\rho : \pi_1(\T^2 - \{p_1, \dots , p_n\} x_0) \to \Z/2\Z
\end{equation}
such that $\rho(m) = \rho (p) = \bar 0$ and $\rho(\alpha_j) = \bar 1$ for every $j= 1, \dots, n$. Let $\pi_0 : \Sigma \to \T^2$ be the 
branched covering associated to \eqref{branchrep}, which can be also constructed as follows: 
\begin{list}{\labelitemi}{\leftmargin=17pt}

\item[(1)] we choose a small closed disk $D_j$ centered at $p_j$, 

\item[(2)] we take the 2-fold branched covering $\pi_0 : D'_j \to D_j$ with branch point $p_j$ which is associated to the representation $\rho_j$ 
sending the generator $\alpha_j$ of $\pi_1(D_j-\{p_j\})$ to the generator of $\Z/2\Z$ (see Figure~\ref{fig:branchedD}),

\item[(3)] we remove the interiors of the disks $D_j$ and we glue together the 2-fold regular covering of $\T^2 - \bigsqcup_{j=1}^n \inte{D}_j$  defined by the representation~\eqref{branchrep} 
and the 2-fold covers $D'_j$ using a diffeomorphism from each $\partial D'_j$ to the 2-fold cover of  each $\partial D_j$.
\end{list}

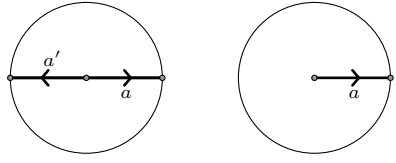
\begin{figure}[t]
\definecolor{qqttcc}{rgb}{0.6,0.6,0.6}
\begin{tikzpicture}[line cap=round,line join=round,>=triangle 45,x=1.0cm,y=1.0cm, scale=0.5]
\clip(-3,-2) rectangle (9,3);
\draw(0,0) circle (2cm);
\draw [line width=1pt] (0,0)-- (2,0);
\draw [line width=1pt] (1.17,0) -- (1,-0.2);
\draw [line width=1pt] (1.17,0) -- (1,0.2);
\draw [line width=1pt] (0,0)-- (-2,0);
\draw [line width=1pt] (-1.17,0) -- (-1,0.2);
\draw [line width=1pt] (-1.17,0) -- (-1,-0.2);
\draw(6,0) circle (2cm);
\draw [line width=1pt] (6,0)-- (8,0);
\draw [line width=1pt] (7.17,0) -- (7,-0.2);
\draw [line width=1pt] (7.17,0) -- (7,0.2);
\draw [line width=1pt] (0,0)-- (2,0);
\draw [line width=1pt] (0,0)-- (-2,0);
\begin{scriptsize}
\draw [fill=qqttcc] (0,0) circle (2pt);
\draw [fill=qqttcc]  (2,0) circle (2pt);
\draw[color=black] (1.05,-0.42) node {$a$};
\draw [fill=qqttcc] (-2,0) circle (2pt);
\draw[color=black] (-0.89,0.52) node {$a'$};
\draw [fill=qqttcc]  (6,0) circle (2pt);
\draw [fill=qqttcc] (8,0) circle (2pt);
\draw[color=black] (7.05,-0.42) node {$a$};
\end{scriptsize}
\end{tikzpicture}
\caption{2-fold branched covering of $D$}
\label{fig:branchedD}
\end{figure}

\noindent
The Riemann-Hurwitz formula tells us that
$\chi(\Sigma) = 2 \chi(\T^2) - \sum_{j=1}^n (e_j - 1) = 2 - 2g$,  
where $e_j = 2$ is the index of the unique ramification point over $p_j$. By 
the path lifting property, since the isomorphism induced by $A^m$ preserves the kernel of $\rho$, $A^m$ lifts to pseudo-Anosov homeomorphism $\varphi : \Sigma \to \Sigma$ with the same stretch factor. 
\medskip 
 Now, as in the proof of Theorem~\ref{thm:Franks}, we consider an isotopy between $\O_j$ and a closed transversal $\gamma_j$ to both $f_0$ and $\F_0$ and we construct a 2-fold branched cover $M$ of $\T^3_A$ with branch locus $\gamma = \bigsqcup_{j=1}^n \gamma_j$ as follows:

\begin{list}{\labelitemi}{\leftmargin=17pt}

\item[(1)] we choose a tubular neighborhood $\psi_j : D_j \times \Sp^1 \to V_j$ of 
$\gamma_j = \psi(\{0\} \times \Sp^1)$ such that 
 $D_{p_j} = \psi_j(D_j \times \{z\})$ is a disk in 
the leaf of $\F_0$ passing through $p_j = \psi_j(0,z)$,

\item[(2)] we take a 2-fold branched covering $\pi_0 \times id : D'_j \times \Sp^1 \to D_j \times \Sp^1$ 
with branch locus $\gamma_j$ which is defined by the representation sending the generator $\alpha_j$ of the 
first direct summand of $\pi_1 \big( (D_j-\{p_j\}) \times \Sp^1 \big)\cong \pi_1(D_j-\{p_j\}) \oplus \Z $ 
to the generator of $\Z/2\Z$ and the generator of the second direct summand to $0$,

\item[(3)] we remove the interiors of the tubular neighborhoods $V_j$ and we glue together the 
2-fold regular cover of $\T^3_A - \bigsqcup_{j=1}^n \inte{V}_j$  
defined by the trivial extension of the representation~\eqref{branchrep} 
and the 2-fold covers $D'_j \times \Sp^1$ using a diffeomorphism from each $\partial D'_j \times \Sp^1$ 
to the 2-fold cover of each  $\partial V_j \cong \partial D_j \times \Sp^1$.
\end{list}

\noindent
By construction, $M$ is homeomorphic to $M_\varphi$ and $\F_0$ lifts to foliation $\F$ of $M$ whose leaves are 2-fold branched coverings of the leaves of $\F_0$ with countably many branch points. Reasoning as in the proof of Theorem~\ref{thm:Franks2},  no leaf has a planar end. 

\subsection{
Four branch points}

Let $d$ and $a_1,a_2,a_3,a_4$ be integers such that 
\begin{equation} \label{pillownumbers}
0 < a_j \leq d~, \qquad gcd(d,a_1,a_2,a_3,a_4) = 1~, \qquad \sum_{j=1}^4 a_j = 0\  (\mbox{mod } d).
\end{equation}
Given four distinct points $z_1, z_2, z_3, z_4 \in \C$, we consider the Riemann surface 
$$
\Sigma_d(a_1,a_2,a_3,a_4) = \{ \, (z,w) \in \C^2 \, | \, w^d = (z-z_1)^{a_1}(z-z_2)^{a_2}(z-z_3)^{a_3}(z-z_4)^{a_4} \, \}.
$$
This is a branched cover of the Riemann sphere $\C P^1$ 
where the covering map is the projection onto the $z$-factor and the branch points are $z_1, z_2, z_3, z_4$. The group of deck transformations is the cyclic 
group  $\Z/d\Z$ formed by the transformations $\tau(z,w) = (z,\zeta w)$ where 
$\zeta$ is a primitive $d$-th root of unity.
As explained in \cite{Forni&al}, up to a  factor, the  
quadratic differential 
$q = (dz)^2 / (z-z_1)(z-z_2)(z-z_3)(z-z_4)$
defines a flat structure on $\C P^1$ with singularities at $z_1, z_2, z_3, z_4$, which can be obtained by identifying the boundaries of two squares when points $z_j$ are well chosen, for example $z_j =i^j$, see Figure~\ref{fig:pillow}. 
Thus, the total space of the branched covering $\pi : \Sigma_d(a_1,a_2,a_3,a_4) \to \C P^1$ is a {\em square-tiled surface}, which contains $gcd(d,a_j)$ ramification points over each point $z_j$ with index $d/gcd(d,a_j)$. 
\medskip 

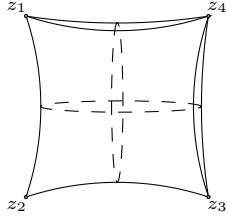
\begin{figure}
\definecolor{qqqqff}{rgb}{0.5,0.5,1.}
\begin{tikzpicture}[line cap=round,line join=round,>=triangle 45,x=1.0cm,y=1.0cm, scale=0.3]
\clip(-6,-5) rectangle (6,6);
\draw [shift={(0,16)}] plot[domain=4.39:5.03,variable=\t]({1*12.65*cos(\t r)+0*12.65*sin(\t r)},{0*12.65*cos(\t r)+1*12.65*sin(\t r)});
\draw [shift={(-16,0)}] plot[domain=-0.32:0.32,variable=\t]({1*12.65*cos(\t r)+0*12.65*sin(\t r)},{0*12.65*cos(\t r)+1*12.65*sin(\t r)});
\draw [shift={(0,-16)}] plot[domain=1.25:1.89,variable=\t]({1*12.65*cos(\t r)+0*12.65*sin(\t r)},{0*12.65*cos(\t r)+1*12.65*sin(\t r)});
\draw [shift={(15.7,0)}] plot[domain=2.81:3.47,variable=\t]({1*12.36*cos(\t r)+0*12.36*sin(\t r)},{0*12.36*cos(\t r)+1*12.36*sin(\t r)});
\draw [shift={(30,0)}] plot[domain=2.99:3.29,variable=\t]({1*26.31*cos(\t r)+0*26.31*sin(\t r)},{0*26.31*cos(\t r)+1*26.31*sin(\t r)});
\draw [shift={(0,30)}] plot[domain=4.56:4.87,variable=\t]({1*26.31*cos(\t r)+0*26.31*sin(\t r)},{0*26.31*cos(\t r)+1*26.31*sin(\t r)});
\draw [rotate around={0:(0.16,0)},dash pattern=on 5pt off 5pt] (0.16,0) ellipse (3.52cm and 0.25cm);
\draw [rotate around={90:(0,0.18)},dash pattern=on 5pt off 5pt] (0,0.18) ellipse (3.53cm and 0.25cm);
\begin{scriptsize}
\draw [fill=qqqqff] (-4,4) circle (2pt);
\draw [fill=qqqqff] (4,4) circle (2pt);
\draw [fill=qqqqff] (-4,-4) circle (2pt);
\draw [fill=qqqqff]  (4,-4) circle (2pt);
\node (a) at (-4.4,4.4) {$z_1$};
\node (b) at (-4.4,-4.4) {$z_2$};
\node (c) at (4.4,-4.4) {$z_3$};
\node (a) at (4.4,4.4) {$z_4$};
\end{scriptsize}
\end{tikzpicture}
\caption{Flat sphere glued from two squares}
\label{fig:pillow}
\end{figure} 

 If $d$ is even and all $a_j$ are odd, there is a holomorphic form  $\omega$ such that $\omega^2 = \pi^\ast q$ and then $\Sigma_d(a_1,a_2,a_3,a_4)$ 
 is a {\em translation surface}. By the Riemann-Hurwitz formula, the Euler characteristic of $\Sigma_d(a_1,a_2,a_3,a_4)$ is 
 given by
$$
\chi(\Sigma_d(a_1,a_2,a_3,a_4) ) = d \chi(\C P^1) - \sum_{j=1}^4 gcd(d,a_j) \big(d/gcd(d,a_j)- 1 \big) = \sum_{j=1}^4 gcd(d,a_j) - 2d
$$
and hence $\Sigma_d(a_1,a_2,a_3,a_4)$ has genus 
\begin{equation} \label{genuspillow}
g = d+1 - \frac 1 2 \sum_{j=1}^4 gcd(d,a_j). 
\end{equation}
By choosing $a_j$ such that $gcd(d,a_j) =1$,
we deduce from \eqref{genuspillow} that $\Sigma_d(a_1,a_2,a_3,a_4)$ is a 
closed surface of genus $g = d-1$. Then $\Sigma_d(a_1,a_2,a_3,a_4)$ appears as a $d/2$-fold branched cover of $\T^2$ with four branch points and exactly one ramification point over each branch point. For example, the translation surface $\Sigma_4(1,1,1,1)$ represented in Figure~\ref{fig:Wollmilchsau} is a $2$-fold branched cover of $\T^2$ with four branch points, see \cite{Forni&al}.
\medskip 

By construction, the regular branched covering $\pi : \Sigma_d(a_1,a_2,a_3,a_4) \to \C P^1$ is associated to the homomorphism 
$$
\rho : \pi_1( \C P^1 - \{z_1,z_2,z_3,z_4\}, x_0) = | \; \alpha_j : \prod_{j=1}^4 \alpha_j = 1 \; |  \to \Z/d\Z
$$
given by 
$\rho(\alpha_j) = a_j$, where $\alpha_j$ is the element represented by a simple loop (with base point $x_0$) which surrounds $z_j$ counterclockwise. 
Now consider the involution $J : \T^2 \to \T^2$ given by $J(x,y) = (-x,-y)$. This is a hyperelliptic involution with four fixed points $p_1,p_2,p_3,p_4$. The action generated by $J$ yields a $2$-fold branched cover $q : \T^2 \to  \C P^1$ and we can arrange $q(p_j) = z_j$. Then $q$ is associated with the homomorphism 
$$
\phi : \pi_1( \C P^1 - \{z_1,z_2,z_3,z_4\}, x_0) \to \Z/2\Z
$$ 
given by $\phi(\alpha_j) = \bar 1$. Notice that $\phi$ is the composition of the homomorphism $\rho$ with the canonical projection of $\Z/d\Z$ onto $\Z/2\Z$. Therefore the branched covering $\pi$ splits as 
$$
\Sigma_d(a_1,a_2,a_3,a_4) \stackrel{\pi_0}{\longrightarrow} \T^2 \stackrel{q}{\longrightarrow} \C P^1.
$$
Now the $d/2$-fold branched covering $\pi_0 : \Sigma_d(a_1,a_2,a_3,a_4) \to \T^2$ is regular, associated to the homomorphism 
\begin{equation} \label{branchrep2}
\rho_0 : \pi_1( \T^2 - \{p_1,p_2,p_3,p_4\}, x_0)  \to Ker(\Z/d\Z \to \Z/2\Z) \cong \Z/(d/2)\Z.
\end{equation}
Given any hyperbolic automorphism $A$ of $\T^2$ which leaves each $p_j$ and the base point $x_0$ fixed, there is a large enough $m$ such that $\rho_0 \scirc A^m_\ast = \rho_0$ since the set 
$$
Hom(\pi_1( \T^2 - \{p_1,p_2,p_3,p_4\}, x_0),\Z/(d/2)\Z)
$$
is finite. It follows that the kernel of $\rho_0$ is preserved by the action of $A^m$ and then $A^m$ lifts to a pseudo-Anosov homeomorphism
$\varphi : \Sigma_d(a_1,a_2,a_3,a_4)  \to \Sigma_d(a_1,a_2,a_3,a_4)$.
\medskip 

%
%
%
%
%

Let $\O_j$ be the periodic orbit of the Anosov flow that is given by the fixed point $p_j$ of $A^m$. There is a isotopy between $\O_j$ and a closed 
transversal $\gamma_j$ to both the fiber bundle over $\Sp^1$ and the foliation $\mathcal F_0$. Let $M$ be the $d/2$-fold branched cover of $\T^3_A$ with branch locus $\gamma = \bigsqcup_{j=1}^4 \gamma_j$ constructed as follows: 

\begin{list}{\labelitemi}{\leftmargin=17pt}

\item[(1)] we choose a tubular neighborhood $\psi_j : D_j \times \Sp^1 \to V_j$ of $\gamma_j$ as before,

\item[(2)] we take a $d/2$-fold branched covering $\pi_0 \times id : D'_j \times \Sp^1 \to D_j \times \Sp^1$ 
where the generator 
of $\pi_1 (D-\{0\})$ is now sent to the generator of  $\Z/(d/2)\Z$,

\item[(3)] we remove the interiors $\inte{V}_j \cong \inte{D}_j \times \Sp^1$ from $M$, we take the $d/2$-fold covering of $M - \inte{V}$ defined by trivial extension of the representation~\eqref{branchrep2}, 
and we attach four copies of the $d/2$-fold branched cover $D'_j \times \Sp^1$ using a diffeomorphism from each $\partial D'_j \times \Sp^1$ to the $d/2$-fold cover of each $\partial V_j \cong \partial D_j \times \Sp^1$.
\end{list}

\noindent
As previously, $\F_0$ turns into a foliation $\F$ on $M$ whose leaves are $d/2$-fold branched coverings of the leaves of $\F_0$ with countably 
many branch points. There are one ramification point with ramification index $d/2$ over each branch point. Riemann-Hurwitz formula tells us 
again that any leaf has infinite genus and no end is planar. 

\begin{figure}
\begin{tikzpicture}[line cap=round,line join=round,>=triangle 45,x=1.0cm,y=1.0cm,scale=0.6]
\clip(-4.4,-0.5) rectangle (10.6,4.5);
\fill[line width=0.8pt,fill=black,fill opacity=0.2] (-2,0) -- (-2,2) -- (-4,2) -- (-4,0) -- cycle;
\fill[line width=0.8pt,fill=black,fill opacity=0.2] (0,0) -- (0,2) -- (-2,2) -- (-2,0) -- cycle;
\fill[line width=0.8pt,fill=black,fill opacity=0.2] (2,0) -- (2,2) -- (0,2) -- (0,0) -- cycle;
\fill[line width=0.8pt,fill=black,fill opacity=0.2] (4,0) -- (4,2) -- (2,2) -- (2,0) -- cycle;
\fill[line width=0.8pt,fill=black,fill opacity=0.2] (4,2) -- (4,4) -- (2,4) -- (2,2) -- cycle;
\fill[line width=0.8pt,fill=black,fill opacity=0.2] (6,2) -- (6,4) -- (4,4) -- (4,2) -- cycle;
\fill[line width=0.8pt,fill=black,fill opacity=0.2] (8,2) -- (8,4) -- (6,4) -- (6,2) -- cycle;
\fill[line width=0.8pt,fill=black,fill opacity=0.2] (10,2) -- (10,4) -- (8,4) -- (8,2) -- cycle;
\draw [line width=0.8pt] (-2,0)-- (-2,2);
\draw [line width=0.8pt] (-2,2)-- (-4,2);
\draw [line width=0.8pt] (-4,2)-- (-4,0);
\draw [line width=0.8pt] (-4,0)-- (-2,0);
\draw [line width=0.8pt] (0,0)-- (0,2);
\draw [line width=0.8pt] (0,2)-- (-2,2);
\draw [line width=0.8pt] (-2,2)-- (-2,0);
\draw [line width=0.8pt] (-2,0)-- (0,0);
\draw [line width=0.8pt] (2,0)-- (2,2);
\draw [line width=0.8pt] (2,2)-- (0,2);
\draw [line width=0.8pt] (0,2)-- (0,0);
\draw [line width=0.8pt] (0,0)-- (2,0);
\draw [line width=0.8pt] (4,0)-- (4,2);
\draw [line width=0.8pt] (4,2)-- (2,2);
\draw [line width=0.8pt] (2,2)-- (2,0);
\draw [line width=0.8pt] (2,0)-- (4,0);
\draw [line width=0.8pt] (4,2)-- (4,4);
\draw [line width=0.8pt] (4,4)-- (2,4);
\draw [line width=0.8pt] (2,4)-- (2,2);
\draw [line width=0.8pt] (2,2)-- (4,2);
\draw [line width=0.8pt] (6,2)-- (6,4);
\draw [line width=0.8pt] (6,4)-- (4,4);
\draw [line width=0.8pt] (4,4)-- (4,2);
\draw [line width=0.8pt] (4,2)-- (6,2);
\draw [line width=0.8pt] (8,2)-- (8,4);
\draw [line width=0.8pt] (8,4)-- (6,4);
\draw [line width=0.8pt] (6,4)-- (6,2);
\draw [line width=0.8pt] (6,2)-- (8,2);
\draw [line width=0.8pt] (10,2)-- (10,4);
\draw [line width=0.8pt] (10,4)-- (8,4);
\draw [line width=0.8pt] (8,4)-- (8,2);
\draw [line width=0.8pt] (8,2)-- (10,2);
\draw [line width=0.8pt] (9.16,2.09)-- (9.03,1.91);
\draw [line width=0.8pt] (-3,0.1)-- (-3,-0.1);
\draw [line width=0.8pt] (-1.05,0.09)-- (-1.05,-0.11);
\draw [line width=0.8pt] (-0.95,0.09)-- (-0.95,-0.11);
\draw [line width=0.8pt] (0.99,0.11)-- (0.99,-0.09);
\draw [line width=0.8pt] (0.92,0.11)-- (0.92,-0.09);
\draw [line width=0.8pt] (1.08,0.11)-- (1.08,-0.09);
\draw [line width=0.8pt] (2.96,0.1)-- (2.96,-0.1);
\draw [line width=0.8pt] (3.04,0.1)-- (3.04,-0.1);
\draw [line width=0.8pt] (3.1,0.1)-- (3.1,-0.1);
\draw [line width=0.8pt] (2.9,0.1)-- (2.9,-0.1);
\draw [line width=0.8pt] (3,4.11)-- (3,3.91);
\draw [line width=0.8pt] (3.09,4.11)-- (3.09,3.91);
\draw [line width=0.8pt] (5.06,4.13)-- (5.06,3.93);
\draw [line width=0.8pt] (6.96,4.11)-- (6.96,3.91);
\draw [line width=0.8pt] (7.03,4.12)-- (7.03,3.92);
\draw [line width=0.8pt] (7.11,4.12)-- (7.11,3.92);
\draw [line width=0.8pt] (6.89,4.11)-- (6.89,3.91);
\draw [line width=0.8pt] (9.03,4.1)-- (9.03,3.9);
\draw [line width=0.8pt] (8.95,4.1)-- (8.95,3.9);
\draw [line width=0.8pt] (9.12,4.1)-- (9.12,3.9);
\draw [line width=0.8pt] (-2.9,2.11)-- (-3.03,1.93);
\draw [line width=0.8pt] (-0.9,2.1)-- (-1.03,1.92);
\draw [line width=0.8pt] (-0.79,2.09)-- (-0.92,1.91);
\draw [line width=0.8pt] (1.14,2.1)-- (1,1.92);
\draw [line width=0.8pt] (1.26,2.09)-- (1.12,1.92);
\draw [line width=0.8pt] (1.01,2.11)-- (0.88,1.93);
\draw [line width=0.8pt] (5.16,2.09)-- (5.02,1.91);
\draw [line width=0.8pt] (5.27,2.08)-- (5.14,1.9);
\draw [line width=0.8pt] (5.04,2.1)-- (4.91,1.92);
\draw [line width=0.8pt] (7.06,2.1)-- (6.93,1.92);
\draw [line width=0.8pt] (7.17,2.09)-- (7.04,1.91);
\end{tikzpicture}
\caption{The translation surface $\Sigma_4(1,1,1,1)$}
\label{fig:Wollmilchsau}
\end{figure}
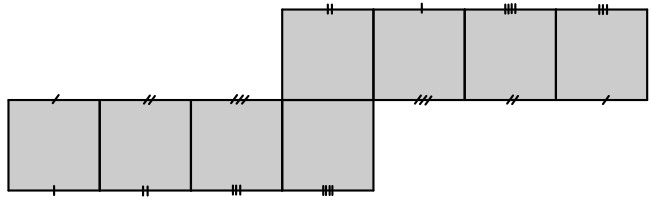

 
\setcounter{equation}{0}

\end{document}